\documentclass[11pt]{article}

\usepackage[a4paper]{geometry}
\usepackage{amsthm}
\usepackage{amsmath}
\usepackage{amssymb}
\usepackage{amsfonts}
\usepackage{graphicx}
\usepackage{color}
\usepackage[bf,SL,BF]{subfigure}
\usepackage{url}

\newcommand{\goodgap}{
  \hspace{\subfigcapskip}}

\usepackage{epsfig,amsbsy,graphicx,multirow}

\usepackage[all]{xypic}

 \numberwithin{equation}{section}
\def\eref#1{(\ref{#1})}

\def\ve{\varepsilon}

\def\f{\varphi}
\def\ve{\varepsilon}

\def\R{\mathbb{R}}

\def\A{\mathcal{A}}

\def\<{\big\langle}
\def\>{\big\rangle}

\def\diiv{\operatorname{div}}
\def\dist{\operatorname{dist}}

\def\f{{\operatorname{-flux}}}

\newtheorem{Lemma}{Lemma}[section]

\newtheorem{Theorem}{Theorem}[section]
\newtheorem{Proposition}{Proposition}[section]

\theoremstyle{remark}
\newtheorem{Remark}{Remark}[section]

\theoremstyle{definition}

\newtheorem{Definition}{Definition}[section]

\theoremstyle{definition}

\begin{document}
\title{Localized bases for finite dimensional homogenization approximations with  non-separated scales and high-contrast.}

\date{\today}

\author{  Houman Owhadi\footnote{Corresponding author. California Institute of
Technology, Computing \& Mathematical Sciences , MC 217-50 Pasadena, CA 91125,
owhadi@caltech.edu}, \quad  Lei Zhang\footnote{University of Oxford, Mathematical Institute}} \maketitle

\begin{abstract}
We construct finite-dimensional approximations of solution spaces of
 divergence form operators with $L^\infty$-coefficients. Our method does not rely on concepts of ergodicity or scale-separation, but on the property that the solution space of these operators is compactly embedded in $H^1$ if source terms are in the unit ball of $L^2$ instead of the unit ball of $H^{-1}$. Approximation spaces are generated by solving elliptic PDEs on localized sub-domains with source terms corresponding to approximation bases for $H^2$. The $H^1$-error estimates show that $\mathcal{O}(h^{-d})$-dimensional spaces with basis elements localized to sub-domains of diameter $\mathcal{O}(h^\alpha \ln \frac{1}{h})$ (with $\alpha \in [\frac{1}{2},1)$) result in an $\mathcal{O}(h^{2-2\alpha})$  accuracy for elliptic, parabolic and hyperbolic problems. For high-contrast media, the accuracy of the method is preserved provided that localized sub-domains contain buffer zones of width  $\mathcal{O}(h^\alpha \ln \frac{1}{h})$ where the contrast of the medium remains bounded.
 The proposed method can naturally be generalized to vectorial equations (such as elasto-dynamics).
\end{abstract}
\maketitle

\section{Introduction}

Consider the partial differential equation
\begin{equation}\label{scalarproblem0}
\begin{cases}
    -\diiv \Big(a(x)  \nabla u(x)\Big)=g(x) \quad  x \in \Omega;\, g \in L^2(\Omega), \;   a(x)=\{a_{ij} \in L^{\infty}(\Omega)\}\\
    u=0 \quad \text{on}\quad \partial \Omega,
    \end{cases}
\end{equation}
where $\Omega$ is a bounded subset of $\R^d$ with a smooth boundary (e.g., $C^2$)
and $a$ is symmetric and uniformly elliptic on $\Omega$. It follows
that the eigenvalues of $a$ are uniformly bounded from below and
above by two strictly positive constants, denoted by $\lambda_{\min}(a)$ and
$\lambda_{\max}(a)$.  Precisely, for all $\xi \in \R^d$ and $x\in \Omega$,
\begin{equation}\label{skjdhskhdkh3e}
\lambda_{\min}(a)|\xi|^2 \leq \xi^T a(x) \xi \leq \lambda_{\max}(a)
|\xi|^2.
\end{equation}

In this paper, we are interested in the homogenization of \eqref{scalarproblem0} (and its parabolic and hyperbolic analogues in Sections \ref{kjsdkjhdkjshkd} and \ref{sechyperbolic}), but not in the classical sense, i.e., that  of asymptotic analysis \cite{BeLiPa78} or that of $G$ or $H$-convergence (\cite{Mur78}, \cite{MR0240443,Gio75}) in which one considers a sequence of operators $-\diiv(a_\epsilon \nabla)$ and seeks to characterize limits of solution. We are interested in the homogenization of \eqref{scalarproblem0} in the sense of ``numerical homogenization,'' i.e., that of the approximation of the solution space of \eqref{scalarproblem0}  by a finite-dimensional space.

This approximation is not based on concepts of scale separation and/or of ergodicity but on compactness properties, i.e., the fact that the unit ball of the solution space is compactly embedded into $H^1_0(\Omega)$ if source terms ($g$) are integrable enough.
This higher integrability condition on $g$ is  necessary because
 if $g$ spans $H^{-1}(\Omega)$, then the solution space of \eqref{scalarproblem0} is $H^1_0(\Omega)$ (and it is not possible to obtain a finite dimensional approximation subspace of $H^1_0(\Omega)$ with arbitrary accuracy in $H^1$-norm). However,
if $g$ spans the unit ball of $L^2(\Omega)$, then the solution space of \eqref{scalarproblem0} shrinks to a compact subset of $H^1_0(\Omega)$ that can be approximated to an arbitrary accuracy in $H^1$-norm by finite-dimensional spaces \cite{BerlyandOwhadi10} (observe that if $a=I_d$, then the solution space is a closed bounded subset of $H^2\cap H^1_0(\Omega)$, which is known to be compactly embedded into $H^1_0(\Omega)$).

The  identification of localized bases spanning accurate approximation spaces relies on a transfer property obtained  in \cite{BerlyandOwhadi10}. For the sake of completeness, we will give a short reminder of that property in Section \ref{transferproperty}.
In Section \ref{jshgjhgdshjghdjhe}, we will construct localized approximation bases with rigorous error estimates (under no further assumptions on $a$ than those given above).
In Sub-section \ref{highcontrast}, we will also address the high-contrast scenario in which $\lambda_{\max}(a)$ is allowed to be large.
In Sections \ref{kjsdkjhdkjshkd} and \ref{sechyperbolic}, we will show that the approximation spaces obtained by solving localized elliptic PDEs remain accurate for parabolic and hyperbolic time-dependent  problems. We refer to Section \ref{numexperiments} for numerical experiments. We refer to Section \ref{compactness} of the Appendix for further discussion and a proof of the strong compactness of the solution space when the range of $g$ is a closed bounded subset of $H^{-\nu}(\Omega)$ with $\nu<1$ (this notion of strong compactness constitutes a simple but fundamental link between classical homogenization, numerical homogenization and reduced order modeling).

\section{A reminder on the flux-norm and the transfer property.}\label{transferproperty}
Recall that the key element in $G$ and $H$ convergence is a notion of ``compactness by compensation'' combined with convergence of fluxes.
Here, the notion of compactness is combined with a flux-norm introduced in \cite{BerlyandOwhadi10}.

\paragraph{The flux-norm.} We will now give a short reminder on the flux-norm and its properties.
\begin{Definition}\label{defWeyl}
 For
 $k\in
(L^2(\Omega))^d$, denote by $k_{pot}$
 the potential portion of the Weyl-Helmholtz decomposition of $k$. Recall that $k_{pot}$ is the orthogonal projection of $k$ onto
$\{\nabla f\;:\; f\in H^1_0\Omega)\}$ in $(L^2(\Omega))^d$.
\end{Definition}

\begin{Definition}\label{FluxNorm}
For $\psi \in H^1_0(\Omega)$, define
\begin{equation}\label{lakdlkjkjd3}
\|\psi\|_{a\f}:=\|(a \nabla \psi)_{pot}\|_{(L^2(\Omega))^d}.
\end{equation}
We  call $\|\psi\|_{a\f}$ the flux-norm of $\Psi$.
\end{Definition}
The following proposition  shows that the flux-norm is equivalent to the energy norm if $\lambda_{\min}(a)>0$ and
$\lambda_{\min}(a)<\infty$.

\begin{Proposition}\label{Prop0} {\rm [Proposition 2.1 of \cite{BerlyandOwhadi10}]}
$\|.\|_{a\f}$ is a norm on $H^1_0(\Omega)$. Furthermore, for all $\psi
\in H^1_0(\Omega)$
\begin{equation}\label{skkshg}
\lambda_{\min}(a) \|\nabla \psi\|_{(L^2(\Omega))^d} \leq
\|\psi\|_{a\f}\leq \lambda_{\max}(a) \|\nabla \psi\|_{(L^2(\Omega))^d}.
\end{equation}
\end{Proposition}

\paragraph{Motivations behind the flux-norm:}
There are three main motivations behind the introduction of the flux norm.
\begin{itemize}
\item The flux-norm allows to obtain approximation error estimates independent from both the minimum and maximum eigenvalues of $a$. In fact, the flux-norm of the solution of \eqref{scalarproblem0} is independent from $a$ altogether since
    \begin{equation}\label{skdcrrvrvkshg}
\|u\|_{a\f}=\|\nabla \Delta^{-1} g\|_{(L^2(\Omega))^d}.
\end{equation}
\item The   $(\cdot)_{\text{pot}}$ in the  $a\f$-norm is explained by the fact that in practice, we are interested in fluxes (of heat, stress, oil, pollutant) entering or exiting a given domain. Furthermore, for a vector field $\xi$, $\int_{\partial\Omega}\xi\cdot n ds=\int_{\Omega}\text{div}(\xi_{\text{pot}})dx$, which means that the flux entering or exiting is determined by the potential part of the vector field.
    \item Classical homogenization is associated with two types of convergence: convergence of energies ($\Gamma$-convergence \cite{MR630747, MR1968440}) and  convergence of fluxes ($G$ or $H$-convergence \cite{Mur78,Gio75, MR0477444, MR0240443, MR506997}). Similarly, one can define an energy norm and a flux-norm.
\end{itemize}

\paragraph{The transfer property.}

For $V$, a finite dimensional linear subspace of $H^1_0(\Omega)$, we define
\begin{equation}\label{ksjjseddesel3}
(\diiv a \nabla V):=\operatorname{span} \{\diiv(a\nabla v)\,:\,  v\in V\}.
\end{equation}
Note that $(\diiv a \nabla V)$ is a finite dimensional subspace of $H^{-1}(\Omega)$.
\smallskip

\begin{Theorem} \label{sdjhskjdhskdhkhje}{\bf (Transfer property of the flux norm)} {\rm [Theorem 2.1 of \cite{BerlyandOwhadi10}]}
 Let $V'$ and $V$ be finite-dimensional subspaces of $H^1_0(\Omega)$. For $f\in L^2(\Omega)$, let $u$ be the solution of
\eref{scalarproblem0} with conductivity $a$ and  $u'$ be the solution of
\eref{scalarproblem0} with conductivity $a'$. If $(\diiv a \nabla V)=(\diiv a' \nabla V')$, then
\begin{equation}\label{sidasasedsssddaud}
 \inf_{v\in V}  \frac{ \|u - v\|_{a\f}}{
\|g\|_{L^2(\Omega)}}= \inf_{v\in V'}  \frac{ \|u' - v\|_{a'\f}}{
\|g\|_{L^2(\Omega)}}.
\end{equation}
\end{Theorem}
The usefulness of \eqref{sidasasedsssddaud} can be illustrated by considering $a'=I$ so that
$\diiv a' \nabla = \Delta$. Then, $u' \in H^2$ and therefore  $V'$  can be chosen as, e.g.,  the standard piecewise linear  FEM space, on a regular triangulation of $\Omega$ of resolution $h$, with nodal basis $\{\phi_i\}$. The space $V$ is then defined by its basis $\{\theta_i\}$ determined by
\begin{equation}\label{lakjlkajlei23}
\begin{cases}
\diiv (a \nabla \theta_i) = \Delta \phi_i \quad &\text{in}\quad \Omega\\
\theta_i=0 &\text{on} \quad \partial \Omega.
\end{cases}
\end{equation}
 Equation \eqref{sidasasedsssddaud}  shows that  the approximation error estimate  associated with the space $V$ and the problem with arbitrarily rough coefficients is (in $a$-flux norm)  equal to  the approximation error estimate  associated with piecewise linear elements and the space $H^2(\Omega)$. More precisely,
\begin{equation}\label{sidasasedsssddaudjo}
\sup_{g \in L^2(\Omega)} \inf_{v\in V}  \frac{ \|u - v\|_{a\f}}{
\|g\|_{L^2(\Omega)}}\leq C h,
\end{equation}
where $C$ does not depend on $a$.

We refer to  \cite{ChuHou09}, \cite{EffGaWu09} and \cite{MR1771781} for recent results on  finite element methods for high contrast ($\lambda_{\max}(a)/\lambda_{\min}(a)>>1$) but non-degenerate ($\lambda_{\min}(a)=\mathcal{O}(1)$) media under specific assumptions on the morphology of the (high-contrast) inclusions (in \cite{ChuHou09}, the mesh has to be adapted to the morphology of the inclusions). Observe that the  proposed method remains accurate if the medium is both of high contrast and degenerate ($\lambda_{\min}(a)<<1$), without any further limitations on $a$, at the cost of solving PDEs \eqref{lakjlkajlei23} over the whole domain $\Omega$.

\begin{Remark}\label{rmkexpconv}
We refer to \cite{BerlyandOwhadi10} for the optimal constant $C$ in \eqref{sidasasedsssddaudjo}.
 This question of optimal approximation with respect to a linear finite dimensional space is related to  the Kolmogorov n-width \cite{Pinkus85, MR1766938}, which measures how accurately a given set of functions can be approximated by linear spaces of dimension $n$ in a given norm. A surprising result of the theory of n-widths is the non-uniqueness of the space realizing the optimal approximation \cite{Pinkus85}.
 Observe also that, as another consequence of the transfer property \eqref{sidasasedsssddaud},  a $h^{k+1}$ rate of convergence can be achieved in \eqref{sidasasedsssddaudjo} by replacing $\phi_i$ with higher-order basis functions in \eqref{lakjlkajlei23}, and $\|g\|_{L^2}$ with $\|g\|_{H^k}$ in \eqref{sidasasedsssddaudjo}. Similarly an exponential rate of convergence can be achieved if the source terms $g$ are analytic. This is the reason behind the near exponential rate of convergence observed in \cite{BaLip10} for harmonic functions (i.e., with zero source terms, and particular ``buffer'' solutions computed near the boundary) and bounded (non high) contrast media.
\end{Remark}

\section{Localization of the transfer property.}\label{jshgjhgdshjghdjhe}

The elliptic PDEs \eqref{lakjlkajlei23} have to be solved on the whole domain $\Omega$. Is it possible to localize the computation of the basis elements $\theta_i$ to a neighborhood of the support of the elements $\phi_i$?
Observe that the support of each $\phi_i$ is contained in a ball $B(x_i, C\,h)$ of center $x_i$ (the node of the coarse mesh associated
with $x_i$) of radius $C\,h$. Let $0<\alpha \leq 1$. Solving the PDEs \eqref{lakjlkajlei23} on sub-domains of $\Omega$ (containing the support of $\phi_i$) may, a priori, increase the error estimate in the right hand side of \eqref{sidasasedsssddaud}. This increase can, in fact, be linked to the decay of the Green's function of the operator $-\diiv (a\nabla)$. The slower the decay, the larger the degradation of those approximation error estimates. Inspired by the strategy used in \cite{Gloria10} for controlling cell resonance errors in the computation of the effective conductivity of periodic or stochastic homogenization (see also \cite{GloriaOtto10, MR712714, MR867870}), we will replace the operator $-\diiv(a\nabla)$ by the operator $\frac{1}{T}-\diiv(a\nabla)$ in the left hand side of \eqref{lakjlkajlei23} in order to artificially introduce an exponential decay in the Green's function. A fine tuning of $T$ is required because although a decrease in $T$ improves the decay of the Green function, it also deteriorates the accuracy of the transfer property.
In order to limit this deterioration, we will transfer a vector space with a higher approximation order than the one associated with piecewise linear elements. Let us now give the main result.

\subsection{Localized bases functions.}\label{jhsgdjhdghgdhed}
Let $h\in (0,1)$. Let $X_h$ be an approximation sub-vector space of $H^1_0(\Omega)$ such that
\begin{itemize}
\item $X_h$ is spanned by basis functions $(\varphi_i)_{1\leq i \leq N}$ (with $N=\mathcal{O}(|\Omega|/ h^d)$) with supports in $B(x_i,C \,h)$ where, the $x_i$ are the nodes of a regular triangulation of $\Omega$ of resolution $h$.
\item $X_h$ satisfies the following approximation properties:
For all $f\in H^1_0(\Omega)\cap H^2(\Omega)$
\begin{equation}\label{Xprop1}
\inf_{v \in X_h} \|f-v\|_{H^1_0(\Omega)}\leq C \,h \|f\|_{H^2(\Omega)},
\end{equation}
and for all $f\in H^1_0(\Omega)\cap H^3(\Omega)$
\begin{equation}\label{Xprop2}
\inf_{v \in X_h} \|f-v\|_{H^1_0(\Omega)}\leq C \,h^2 \|f\|_{H^3(\Omega)}.
\end{equation}
\item For all $i$,
\begin{equation}\label{Controlphi}
\int_{\Omega} |\nabla \varphi_i|^2 \leq  C h^{d-2}.
\end{equation}

\item For all coefficients $c_i$,
\begin{equation}\label{Controlphi2sum}
h^{d} \sum_i   c_i^2  \leq   C  \|\sum_{i} c_i \nabla \varphi_i\|_{L^2(\Omega)}^2.
\end{equation}
\end{itemize}
\begin{Remark}
Examples of such spaces can be found in \cite{BrSc02} and constructed using piecewise quadratic polynomials. From the first bullet point it follows that $h$ can be though of as the diameter of the support of the elements $\varphi_i$.
The largest parameter $h^{d}/C$ satisfying \eqref{Controlphi2sum} is the minimal eigenvalue of the stiffness matrix $(\int_{\Omega} (\nabla \varphi_i)^T \nabla \varphi_j )_{1\leq i,j\leq N}$ and Condition \eqref{Controlphi2sum} is obtained   from the regularity of the tessellation of $\Omega$.
In fact,  the proof of Proposition \ref{jhgsjhgdg6354r5df} shows that Condition \eqref{Controlphi2sum}  can be relaxed to
the assumption of existence of a constant $d_{\varphi}>0$ independent from $h$ such that for all coefficients $c_i$
\begin{equation}\label{Controlphi2sumalternate}
h^{d_{\varphi}} \sum_i   c_i^2  \leq   C  \|\sum_{i} c_i \nabla \varphi_i\|_{L^2(\Omega)}^2.
\end{equation}
\end{Remark}

Through this paper, we will write $C$ any constant that does not depend on $h$ (but may depend on $d$, $\Omega$, and the essential supremum and infimum of the maximum and minimum eigenvalues of $a$ over $\Omega$).
Let $\alpha \in (0,1)$ and $C_1>0$. For each basis element $\varphi_i$ of $X_h$ let $\psi_i$ be the solution of
\begin{equation}\label{loclakjlkajlei23}
\begin{cases}
h^{-2\alpha}\psi_i-\diiv (a \nabla \psi_i) = \Delta \varphi_i \quad &\text{in}\quad B(x_i,C_{1}  h^\alpha \ln \frac{1}{h})\cap \Omega\\
\psi_i=0 &\text{on} \quad \partial \big(B(x_i,C_1 h^\alpha \ln \frac{1}{h})\cap \Omega\big).
\end{cases}
\end{equation}
Let
\begin{equation}\label{Vh}
V_h:=\operatorname{span}(\psi_i)
\end{equation}
be the linear space spanned by the elements $\psi_i$.
\begin{Theorem}\label{thnloc}
For $g\in L^2(\Omega)$, let $u$ be the solution of \eqref{scalarproblem0} in $H^1_0(\Omega)$ and $u_h$ the solution of \eqref{scalarproblem0} in $V_h$.
There exists $C_0>0$ such that for $C_1 \geq C_0$, we have
\begin{equation}\label{ksahskjdhkhwsj}
\frac{\|u-u_h\|_{H^1_0(\Omega)}}{\|g\|_{L^2(\Omega)}}\leq \begin{cases} C h \quad &\text{if}\quad \alpha \in (0,\frac{1}{2}]\\
 C h^{2-2\alpha} \quad &\text{if}\quad \alpha \in [\frac{1}{2},1), \end{cases}
\end{equation}
where the constants $C$ and $C_0$ depend on $a$, $d$, $\Omega$ but not on $h$.
\end{Theorem}
\begin{Remark}
Theorem \ref{thnloc} shows the convergence rate in approximation error remains optimal (i.e., proportional to $h$) after  localization  if $0<\alpha \leq 1/2$ and decays to $0$ as $h^{2-2\alpha}$ for $\frac{1}{2}\leq \alpha <1$. In particular, choosing localized domains with radii  $\mathcal{O}(\sqrt{h} \ln \frac{1}{h})$ is sufficient to obtain the optimal convergence rate $\mathcal{O}(h)$. Observe that the choice of the constant $\alpha$ in equation \eqref{loclakjlkajlei23} is arbitrary.
\end{Remark}
\begin{Remark}
According to Theorem \ref{thnloc}, the constant $C_1$ in \eqref{loclakjlkajlei23} needs to be chosen larger than $C_0$ to achieve the convergence rate $h+h^{2-2\alpha}$. The constant $C_0$ depends on $\alpha$, $d$, $\lambda_{\min}(a)$ and $\lambda_{\max}(a)$.  The constant $C$ in the right hand side of \eqref{ksahskjdhkhwsj} also depends on $\alpha$, $d$, $\lambda_{\min}(a)$ and $\lambda_{\max}(a)$. It is possible to give an explicit value for $C_0$ and $C$ by tracking constants in the proof (in particular, as stated in Subsection \ref{highcontrast}, the dependence on $\lambda_{\max}(a)$ can be removed if the elements $\Psi_i$ are computed on sub-domains with added buffer zones around high-conductivity inclusions).
\end{Remark}
\begin{Remark}
If one uses piecewise linear basis elements instead of the elements $\varphi_i$ (i.e., in the absence of property \eqref{Xprop2}), then the estimate in the right hand side of \eqref{ksahskjdhkhwsj} deteriorates to $h^{1-2\alpha}$. The proof of this remark is similar to that of Theorem \ref{thnloc}. The main modification lies in replacing $h^2/T$ by $h/T$ in equations \eqref{hhgsfhdgfhfd} and \eqref{jadeesehjhguey3}.
\end{Remark}
\begin{Remark}
One could use piecewise linear basis elements instead of the elements $\varphi_i$, and also remove
the term $h^{-2\alpha}\psi_i$ from the transfer property \eqref{loclakjlkajlei23}. In this situation, we numerically observe a rate of convergence of $h$ for periodic, stochastic and low-contrast media after localization of \eqref{loclakjlkajlei23} to balls of radii $\mathcal{O}(h)$.
 In these \emph{particular} situations  (characterized by short range correlations in $a$), the term $h^{-2\alpha}\psi_i$
should be avoided to obtain the optimal convergence rate $h$ after localization to sub-domains of size $\mathcal{O}(h)$. In that sense,  the estimate in the right hand side of \eqref{ksahskjdhkhwsj} corresponds to a \emph{worst case scenario} with respect to the medium $a$ (characterized by long range correlations), requiring the introduction of the term $h^{-1}\psi_i$ and a localization to sub-domains of size $\mathcal{O}(\sqrt{h} \ln \frac{1}{h})$ for the optimal convergence rate $h$.
\end{Remark}

\begin{Remark}
For the elliptic problem, computational gains result from localization (the elements $\psi_i$ are computed on sub-domains $\Omega_i$ of $\Omega$), parallelization (the elements $\psi_i$ can be computed independently from each other), and the fact that the same basis can be used for different right hand sides $g$ in \eqref{scalarproblem0}.
Computational gains are even more significant for time-dependent problems because,  once an accurate basis has been determined for the elliptic problem, the same  basis  can be used for the associated (parabolic and hyperbolic) time-dependent problems with the same accuracy (we refer to Sections \ref{kjsdkjhdkjshkd}  and \ref{sechyperbolic}). For the wave equation with rough bulk modulus and density coefficients,
the proposed method (based on pre-computing basis elements as solutions of localized elliptic PDEs) remains accurate, provided that high frequencies are not strongly excited ($\partial_t g \in L^2$).
\end{Remark}

\paragraph{On Localization.}
We refer to  \cite{ChuHou09}, \cite{EffGaWu09} and \cite{BaLip10} for recent localization results for divergence-form elliptic PDEs.
The strategy of \cite{ChuHou09} is to construct triangulations and finite element bases that are adapted to the shape of high conductivity inclusions via coefficient dependent boundary conditions for the subgrid problems (assuming $a$ to be piecewise constant and the number of inclusions bounded).
The strategy of \cite{EffGaWu09} is to solve local eigenvalue problems, observing that only a few eigenvectors are sufficient to obtain a good pre-conditioner. Both \cite{ChuHou09} and \cite{EffGaWu09}  require specific assumptions on the morphology and number of inclusions. The idea of the strategy is to observe that if $a$ is piecewise constant and the number of inclusions bounded, then $u$ is locally $H^2$ away from the interfaces of the inclusions. The inclusions can then be taken care of by adapting the mesh and the boundary values of localized problems or by observing that those inclusions will affect only a finite number of eigenvectors.

The strategy of \cite{BaLip10} is to construct Generalized Finite Elements
 by partitioning the computational domain  into to a collection of preselected subsets and compute optimal local bases (using the concept of $n$-widths \cite{MR774404}) for the approximation of harmonic functions. Local bases are constructed by solving local eigenvalue problems (corresponding to computing eigenvectors of $P^* P$ where $P$ is the restriction of $a$-harmonic functions from $\omega^*$ onto $\omega\subset \omega^*$, $P^*$ is the adjoint of $P$, and $\omega$ is a sub-domain of $\Omega$ surrounded by a larger sub-domain $\omega^*$).
 The method proposed in \cite{BaLip10} achieves a near exponential convergence rate (in the number of pre-computed bases functions) for harmonic functions. Non-zero right hand sides ($g$) are then taken care of by solving (for each different $g$) particular solutions on preselected subsets with a constant  Neumann boundary condition (determined according to the consistency condition).

As explained in Remark \ref{rmkexpconv}, the near exponential rate of convergence observed in \cite{BaLip10} is explained by the fact that the source space considered in \cite{BaLip10} is more regular than $L^2$ (since \cite{BaLip10}  requires the computation particular (local) solutions for each right hand sides $g$ and each non-zero boundary conditions, the basis obtained in \cite{BaLip10} is in fact adapted to $a$-harmonic functions away from the boundary). The strategy proposed here can also be used to achieve exponential convergence for analytic source terms $g$ by employing higher-order basis functions $\varphi_i$ in \eqref{loclakjlkajlei23}.
Furthermore, as shown in sections \ref{kjsdkjhdkjshkd}, \ref{sechyperbolic} and \ref{highcontrast} the method proposed here allows for the numerical homogenization of time-dependent problems (because it does not require the computation of particular solutions for different source or boundary terms) and can be extended to high-contrast media.
We also note that the basis functions $\psi_i$ are simpler and cheaper to compute  (equation \eqref{loclakjlkajlei23}) than the eigenvectors
of $P^* P$ required by \cite{BaLip10}. We refer to page 16 of  \cite{BaLip10} for a discussion on the cost of this added complexity.

\subsection{On Numerical Homogenization.}

By now, the field of numerical homogenization has become large enough that it is not possible to give an exhaustive review in this short paper.
Therefore, we will restrict our attention to
works directly related to our work.

-  The multi-scale finite element method \cite{MR1455261, MR1898136, MR1642758} can be seen as a numerical generalization of this idea of oscillating test functions found in $H$-convergence. A convergence analysis for periodic media revealed a resonance error introduced by the microscopic boundary condition \cite{MR1455261, MR1642758}. An over-sampling technique was proposed to  reduce the resonance error \cite{MR1455261}.

- Harmonic coordinates play an important role in various homogenization approaches, both theoretical and numerical. These coordinates  were  introduced in \cite{MR542557} in the context of random homogenization. Next, harmonic coordinates have been used in one-dimensional and quasi-one-dimensional divergence form elliptic problems  \cite{ BabOsb83, MR1286212}, allowing for efficient finite
dimensional approximations. The connection of these coordinates with classical homogenization is made explicit in \cite{AllBri05} in the context of multi-scale finite element methods.
The idea of using particular solutions  in numerical homogenization to approximate the solution space of \eqref{scalarproblem0} appears to  have been first proposed in reservoir modeling in the 1980s \cite{BraWu09}, \cite{WhHo87} (in which a global scale-up method was introduced based on generic flow solutions i.e., flows calculated from generic boundary conditions). Its rigorous mathematical analysis was done
only recently \cite{MR2292954} and is based on the fact that solutions are in fact $H^2$-regular with respect to harmonic coordinates  (recall that they are $H^1$-regular with respect to Euclidean coordinates). The main message here is that if the right hand side of \eqref{scalarproblem0} is in $L^2$, then solutions can be approximated at small scales (in $H^1$-norm) by linear combinations of $d$ (linearly independent) particular solutions ($d$ being the dimension of the space). In that sense, harmonic coordinates are only good candidates for being $d$ linearly independent particular solutions.

The idea of a global change of coordinates analogous to harmonic coordinates has been implemented numerically in order to up-scale porous media flows \cite{MR2322432, MR2281625, BraWu09}. We refer, in particular, to a recent review article \cite{BraWu09}   for an overview of some main challenges in reservoir modeling and a description of global scale-up strategies based on generic flows.

- In  \cite{MR2314852, EnSou08}, the structure of the
medium is numerically decomposed into a micro-scale and a
macro-scale (meso-scale) and solutions of cell problems are computed
on the micro-scale, providing local homogenized matrices that are
transferred (up-scaled) to the macro-scale grid. This procedure
allows one to obtain rigorous homogenization results with controlled
error estimates for non-periodic media of the form
$a(x,\frac{x}{\epsilon})$ (where $a(x,y)$ is assumed to be smooth in
$x$ and periodic or ergodic with specific mixing properties in $y$).
Moreover, it is shown that the numerical algorithms associated with HMM and MsFEM can be
implemented for a class of coefficients that is much broader than $a(x,\frac{x}{\epsilon})$. We refer to \cite{AnGlo06} for convergence results on the Heterogeneous Multiscale Method in the framework of $G$ and $\Gamma$-convergence.

- More recent work includes an adaptive projection based method
\cite{MR2399542}, which is consistent with homogenization when there is scale
separation, leading to adaptive algorithms for solving problems with
no clear scale separation;  fast and sparse chaos approximations of
elliptic problems with stochastic coefficients \cite{MR2317004,
MR2399150, DooOwh10}; finite difference approximations of fully nonlinear,
uniformly elliptic PDEs with Lipschitz continuous viscosity
solutions \cite{MR2361302} and operator splitting methods
\cite{MR2342991, MR2231859}.

- We  refer to \cite{MR2334772, MR2284699} (and references therein) for most recent results on homogenization of  scalar divergence-form elliptic operators with stochastic coefficients. Here, the stochastic coefficients $a(x /\ve, \omega)$ are obtained from stochastic deformations (using random diffeomorphisms) of the periodic and stationary ergodic setting.

\subsection{Proof of Theorem \ref{thnloc}.}

For each basis element $\varphi_i$ of $X_h$, let $\psi_{i,T}$ be the solution of
\begin{equation}\label{loclakjlkajlei23TT}
\begin{cases}
\frac{1}{T}\psi_{i,T}-\diiv (a \nabla \psi_{i,T}) = \Delta \varphi_i \quad &\text{in}\quad  \Omega\\
\psi_{i,T}=0 &\text{on} \quad \partial \Omega.
\end{cases}
\end{equation}
The following Proposition will allow us to control the impact of the introduction of the term $\frac{1}{T}$ in the transfer property.
Observe that the domain of PDE \eqref{loclakjlkajlei23TT} is still $\Omega$ (our next step will be to localize it to $\Omega_i \subset \Omega$).

\begin{Proposition}\label{jhjhejheg}
For $g\in L^2(\Omega)$ let $u$ be the solution of \eqref{scalarproblem0} in $H^1_0(\Omega)$.
Then, there exists $v \in \operatorname{span}(\psi_{i,T})$ such that
\begin{equation}\label{hhgsfhdgfhfd}
 \frac{\|u-v\|_{H^1_0(\Omega)}}{\|g\|_{L^2(\Omega)}}\leq C\big(h+\frac{h^2}{T}\big).
\end{equation}
Furthermore, writing $v:=\sum_i c_i \psi_{i,T}$ we have
\begin{equation}\label{eqci2}
\sum_i c_i^2 \leq  C h^{-d} (1+T^{-2}) \|g\|_{L^2(\Omega)}^2
\end{equation}
\end{Proposition}
\begin{proof}
Let $v=\sum_i c_i \psi_{i,T}$.
We have
\begin{equation}\label{jahjhguey3}
\frac{u-v}{T}-\diiv\big(a\nabla (u-v)\big)=g+\frac{u}{T}-\sum_i c_i \Delta \varphi_i.
\end{equation}
Define $a[v]$ to be the energy norm $a[v]:=\int_\Omega (\nabla v)^T a \nabla v$.
Multiplying \eqref{jahjhguey3} by $u-v$ and integrating by parts, we obtain that
\begin{equation}\label{jadsehjhguey3}
\frac{\|u-v\|_{L^2(\Omega)}^2}{T}+a[u-v]=\int_\Omega (u-v)(g+\frac{u}{T}-\sum_i c_i \Delta \varphi_i).
\end{equation}
Write $c_{i}=c_{i,1}+c_{i,2}$ and let $w_1$ and $w_2$ be the solutions of $\Delta w_1=g-\sum_i c_{i,1} \Delta \varphi_i$ and
$\Delta w_2=\frac{u}{T}-\sum_i c_{i,2} \Delta \varphi_i$ with Dirichlet boundary conditions on $\partial \Omega$. Then, we obtain by integration by parts and
the Cauchy-Schwartz inequality that
\begin{equation}\label{jadsehjhguey3uuuu}
\frac{\|u-v\|_{L^2(\Omega)}^2}{T}+a[u-v]\leq \big\|\nabla(u-v)\big\|_{(L^2(\Omega))^d}\big(\|\nabla w_1\|_{(L^2(\Omega))^d}+\|\nabla w_2\|_{(L^2(\Omega))^d}\big).
\end{equation}
Using \eqref{Xprop1}, we can choose $(c_{i,1})$ so that
\begin{equation}\label{jadseshjhguey3}
\|\nabla w_1\|_{(L^2(\Omega))^d}\leq C h \|g\|_{L^2(\Omega)}.
\end{equation}
Using \eqref{Xprop2}, we can choose $(c_{i,2})$ so that
\begin{equation}\label{jadeesehjhguey3}
\|\nabla w_2\|_{(L^2(\Omega))^d}\leq C \frac{h^2}{T} \|u\|_{H^1_0(\Omega)},
\end{equation}
we conclude the proof of the approximation \eqref{hhgsfhdgfhfd} by observing that $\|u\|_{H^1_0(\Omega)} \leq C \|g\|_{L^2(\Omega)}$.
Let us now prove Equation \eqref{eqci2}. First, observe that Equation \eqref{Controlphi2sum} and the triangular inequality imply that
\begin{equation}\label{Controlphi2sumygyg}
\big( \sum_i   (c_i)^2\big)^\frac{1}{2}  \leq   C  h^{-\frac{d}{2}} \Big(\|\sum_{i} c_{i,1} \nabla \varphi_i\|_{L^2(\Omega)}+\|\sum_{i} c_{i,2} \nabla \varphi_i\|_{L^2(\Omega)}\Big).
\end{equation}
Next, we obtain from \eqref{jadseshjhguey3} and Poincar\'{e} inequality and
\begin{equation}\label{eqf1}
\|\sum_{i} c_{i,1} \nabla \varphi_i\|_{L^2(\Omega)} \leq C  \|g\|_{L^2(\Omega)}
\end{equation}
and
\begin{equation}\label{eqf2}
\|\sum_{i} c_{i,2} \nabla \varphi_i\|_{L^2(\Omega)} \leq C \frac{1}{T} \|g\|_{L^2(\Omega)}
\end{equation}
We conclude by combining equations \eqref{eqf1} and \eqref{eqf2} with \eqref{Controlphi2sumygyg}.
\end{proof}

We will now control the error induced by the localization of the elliptic problem \eqref{loclakjlkajlei23TT}. To this end, for each
each basis element $\varphi_i$ of $X_h$ write $S_i$ the intersection of the support of $\varphi_i$ with $\Omega$ and
let $\Omega_i$ be a subset of $\Omega$ containing $S_i$ such that $\dist(S_i,\Omega/\Omega_i) >0$. Let also
$\psi_{i,T,\Omega_i}$ be the solution of
\begin{equation}\label{loclakjlddrfredkajlei23TT}
\begin{cases}
\frac{1}{T}\psi_{i,T,\Omega_i}-\diiv (a \nabla \psi_{i,T,\Omega_i}) = \Delta \varphi_i \quad &\text{in}\quad  \Omega_i\\
\psi_{i,T,\Omega_i}=0 &\text{on} \quad \partial \Omega_i.
\end{cases}
\end{equation}
For $A,B\subset \Omega$, write $d(A,B)$ the Euclidean distance between the sets $A$ and $B$.
\begin{Proposition}\label{jhgsjhgdg6354r5df}
Extending $\psi_{i,T,\Omega_i}$ by $0$ on $\Omega/\Omega_i$ we have
\begin{equation}\label{gjhsfhgsfdhgd}
\big\|\psi_{i,T}-\psi_{i,T,\Omega_i}\big\|_{H^1(\Omega)} \leq \frac{C h^{\frac{d}{2}-1} (T^{-1}+1)}{\big(\dist(S_i,\Omega/\Omega_i)\big)^{d+1}} \exp\Big(-\frac{\dist(S_i,\Omega/\Omega_i)}{C \sqrt{T}}\Big).
\end{equation}
\end{Proposition}
We refer to Section \ref{jhgsjhgd33ge3e} of the Appendix for the proof of Proposition \ref{jhgsjhgdg6354r5df}.

Taking $\Omega_i:=B(x_i,C_1 h^\alpha \ln \frac{1}{h})\cap \Omega$ (we use the particular notation $C_1$ because our proof of accuracy requires that specific constant to be large enough, i.e., larger than a constant depending on the parameter $C$ appearing in the right hand side of \eqref{gjhsfhgsfdhgd} and the parameter $C$ describing the balls $B(x_i,C \,h)$ containing the support of the
 basis functions $(\varphi_i)_{1\leq i \leq N}$ introduced in Subsection \ref{jhsgdjhdghgdhed}) and $T=h^{2\alpha}$ in equation \eqref{gjhsfhgsfdhgd} of Proposition \ref{jhgsjhgdg6354r5df}, we obtain for $C_1$ large enough (but independent from $h$) that
\begin{equation}\label{gjhwedwdehgd}
\big\|\psi_{i,T}-\psi_{i,T,\Omega_i}\big\|_{H^1(\Omega)} \leq C h^{d+1+2\alpha}.
\end{equation}
Let $u$ be the solution of \eqref{scalarproblem0} in $H^1_0(\Omega)$. Using Proposition \ref{jhjhejheg}, we obtain that there exist coefficients $c_i$ such that
\begin{equation}
\big\|u-\sum_{i} c_i \psi_{i,T}\big\|_{H^1_0(\Omega)}\leq C\big(h+h^{2-2\alpha}\big) \|g\|_{L^2(\Omega)}.
\end{equation}
and
\begin{equation}\label{eqci2deduc}
\sum_i c_i^2 \leq  C h^{-d-4 \alpha} \|g\|_{L^2(\Omega)}^2
\end{equation}
Using the triangle inequality, it follows that
\begin{equation}
\big\|u-\sum_{i} c_i \psi_{i,T,\Omega_i}\big\|_{H^1_0(\Omega)} \leq C\big(h+h^{2-2\alpha}\big) \|g\|_{L^2(\Omega)}+\sum_{i} |c_i| \big\|\psi_{i,T}-\psi_{i,T,\Omega_i}\big\|_{H^1(\Omega)},
\end{equation}
whence, from Cauchy-Schwartz inequality,
\begin{equation}\label{jshgdjhgdh}
\begin{split}
\big\|u-\sum_{i} c_i \psi_{i,T,\Omega_i}\big\|_{H^1_0(\Omega)} \leq & C\big(h+h^{2-2\alpha}\big) \|g\|_{L^2(\Omega)}\\&+\big(\sum_{i} |c_i|^2 \big)^\frac{1}{2}\Big(\sum_{i}  \big\|\psi_{i,T}-\psi_{i,T,\Omega_i}\big\|_{H^1(\Omega)}^2 \Big)^\frac{1}{2}.
\end{split}
\end{equation}
Combining \eqref{jshgdjhgdh} with  \eqref{eqci2deduc}, we obtain that
\begin{equation}\label{jsgdjhddgj}
\begin{split}
\big\|u-\sum_{i} c_i \psi_{i,T,\Omega_i}\big\|_{H^1_0(\Omega)} \leq & C\big(h+h^{2-2\alpha}\big) \|g\|_{L^2(\Omega)}\\& +C h^{-\frac{d}{2}-2\alpha}
\|g\|_{L^2(\Omega)} \Big(\sum_{i}  \big\|\psi_{i,T}-\psi_{i,T,\Omega_i}\big\|_{H^1(\Omega)}^2 \Big)^\frac{1}{2}.
\end{split}
\end{equation}
Using \eqref{gjhwedwdehgd} in \eqref{jsgdjhddgj}, we obtain that
\begin{equation}\label{jsisisusgdjhddgj}
\begin{split}
\big\|u-\sum_{i} c_i \psi_{i,T,\Omega_i}\big\|_{H^1_0(\Omega)} \leq C\big(h+h^{2-2\alpha}\big) \|g\|_{L^2(\Omega)}.
\end{split}
\end{equation}
Observe that it is the exponential decay in \eqref{gjhsfhgsfdhgd} that allows us to compensate for the large term on the right hand side of \eqref{jsgdjhddgj} via \eqref{gjhwedwdehgd}.
This concludes the proof of Theorem \ref{thnloc}.

\begin{figure}[httb]
  \begin{center}
    \subfigure
    {\includegraphics[width=0.35\textwidth,height= 0.3\textwidth]{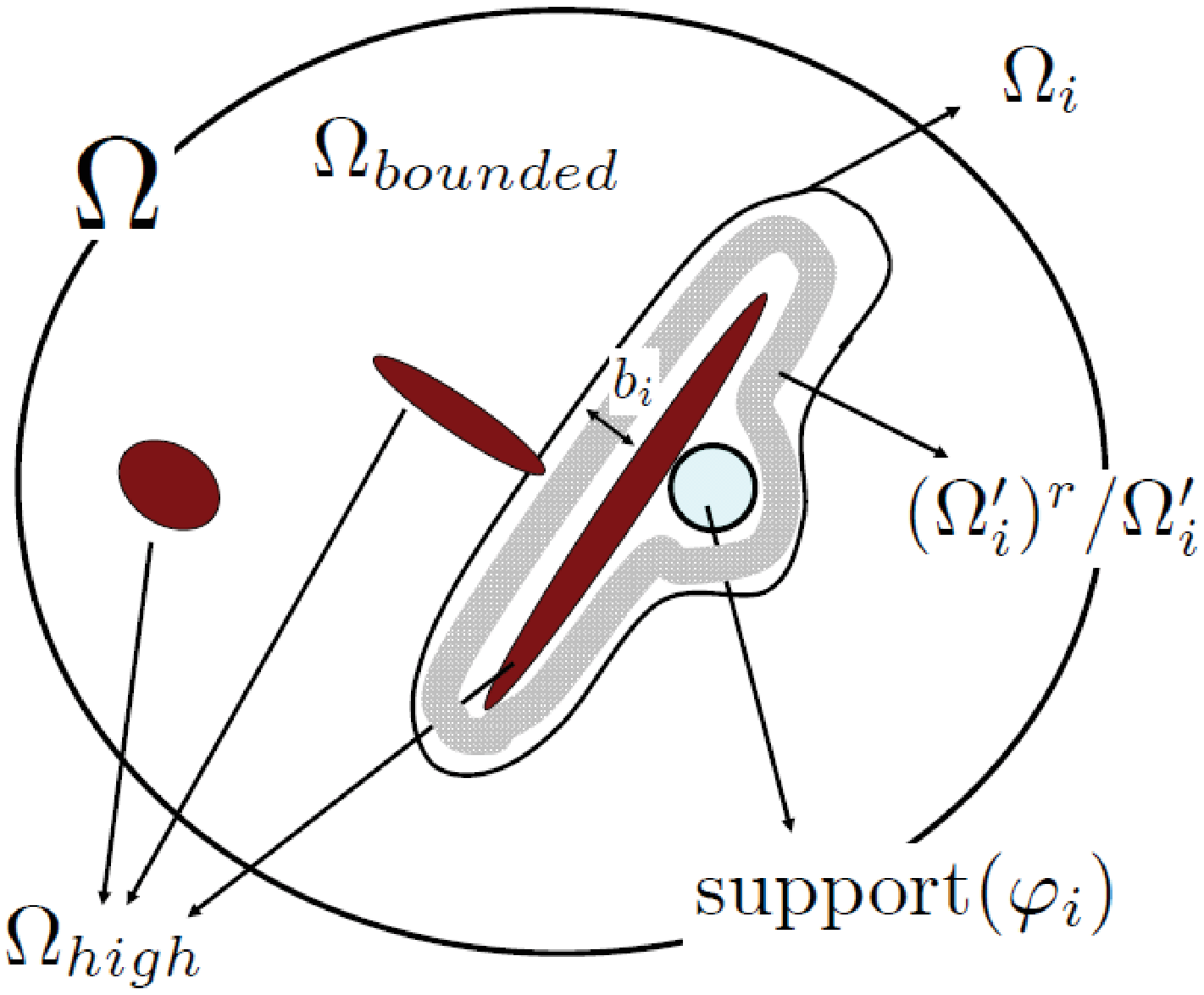}}
    \goodgap
    \subfigure
    {\includegraphics[width=0.35\textwidth,height= 0.3\textwidth]{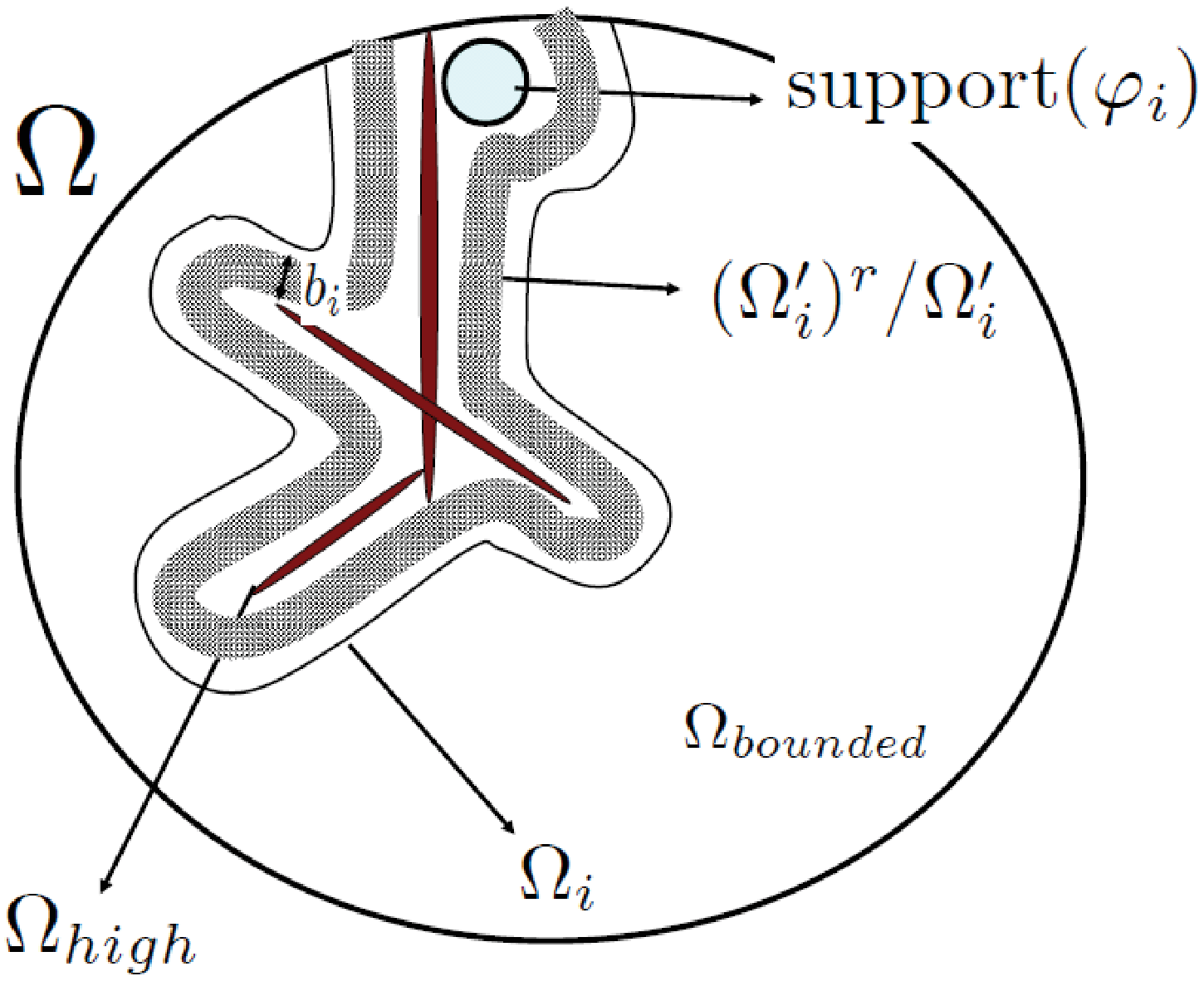}}
    \caption{Illustrations of the buffer distance.}
    \label{fig:bufferillustrate}
  \end{center}
\end{figure}

\subsection{On localization with high-contrast.}\label{highcontrast}
The constant $C$ in the approximation error  estimate \eqref{ksahskjdhkhwsj} depends, a priori, on the contrast of $a$. Is it possible to localize the computation of bases for $V_h$ when the contrast of $a$ is high? The purpose of this subsection is to show that the answer
is yes provided that there is a buffer zone between the boundaries of localization sub-domains and the supports of the elements $\varphi_i$ where the contrast of $a$ remains bounded. More precisely, assume that $\Omega$ is the disjoint union of $\Omega_{bounded}$ and $\Omega_{high}$. Assume that \eqref{skjdhskhdkh3e} holds only on $\Omega_{bounded}$, and that on $\Omega_{high}$ we have
 \begin{equation}\label{skjdsddefhskhdkh3e}
\lambda_{\min}(a)|\xi|^2 \leq \xi^T a(x) \xi \leq \gamma
|\xi|^2.
\end{equation}
where $\gamma$ can be arbitrarily large. Practical examples include media characterized by a bounded contrast background with high conductivity inclusions or channels.  Let $\psi_i^{high}$ be the solution of
\begin{equation}\label{loclakjededwslkajlei23}
\begin{cases}
h^{-2\alpha}\psi_i^{high}-\diiv (a \nabla \psi_i^{high}) = \Delta \varphi_i \quad &\text{in}\quad \Omega_i \\
\psi_i=0 &\text{on} \quad \partial \Omega_i.
\end{cases}
\end{equation}
Let
\begin{equation}\label{Vhprime}
V_h^{high}:=\operatorname{span}(\psi_i^{high})
\end{equation}
be the linear space spanned by the elements $\psi_i^{high}$. For each $i$, define $b_i$ to be the largest number $r$ such that there exists a subset $\Omega_i'$ such that: the closure of $\Omega_i'$ contains the support of $\varphi_i$, $(\Omega_i')^r$ is a subset of $\Omega_i$ (where $A^r$ are the set of points of $\Omega$ that are at distance at most $r$ for $A$), and  $(\Omega_i')^r/\Omega_i'$ is a subset of $\Omega_{bounded}$. If no such subset exists we set $b_i:=0$. $b_i$ can be interpreted as the non-high-contrast buffer distance between the support of $\varphi_i$ and the boundary of $\Omega_i$. We refer to Figure \ref{fig:bufferillustrate} for illustrations of the buffer distance.

\begin{Theorem}\label{thnsslochigh}
For $g\in L^2(\Omega)$, let $u$ be the solution of \eqref{scalarproblem0} in $H^1_0(\Omega)$ and $u_h$ the solution of \eqref{scalarproblem0} in $V_h^{high}$.
There exists $C_0>0$ such that if for all $i$, $b_i\geq C_0 h^\alpha \ln \frac{1}{h}$ then
\begin{equation}\label{ksahskjdhckskjdkjhwsj}
\frac{\|u-u_h\|_{H^1_0(\Omega)}}{\|g\|_{L^2(\Omega)}}\leq \begin{cases} C h \quad &\text{if}\quad \alpha \in (0,\frac{1}{2}]\\
 C h^{2-2\alpha} \quad &\text{if}\quad \alpha \in [\frac{1}{2},1), \end{cases}
\end{equation}
where the constants $C$ and $C_0$ depend on $\lambda_{\min}(a)$, $\lambda_{\max}(a)$ (the bounds on $a$ in $\Omega_{bounded}$), $d$, $\Omega$ but not on $h$ and $\gamma$ (The upper bound on $a$ on $\Omega_{high}$).
\end{Theorem}

\begin{Remark}
Recall that the global basis computed in \eqref{lakjlkajlei23} remains accurate if the medium is both of high contrast ($\lambda_{\max}(a)>>1$) and degenerate ($\lambda_{\min}(a)<<1$). The basis computed in \eqref{loclakjededwslkajlei23} preserves the former property (of accuracy for high contrast media) but loses the latter (property of accuracy in the degenerate case) since the constant $C$ in \eqref{ksahskjdhckskjdkjhwsj} depends on $\lambda_{\min}(a)$.
\end{Remark}

\begin{Remark}
Observe that local solves have to resolve the connected components of high contrast structures.
This is the price to pay for localization with high contrast in the most general case. Recall that in classical homogenization with high contrast the limit of the homogenized operator may be a non-local operator (we refer for instance to \cite{MR2217509}). A similar phenomenon is observed here (distant points connected by high conductivity channels are associated with a low resistance metric and a large coupling coefficient in the numerically homogenized stiffness matrix).
\end{Remark}

The proof of Theorem \ref{thnsslochigh} is similar to that of Theorem \ref{thnloc}, but it requires a precise tracking of the constants involved.
 Because of the close similarity we will not include the proof in this paper but only give its main lines. First, the proof of Proposition \ref{jhjhejheg} remains unchanged as the constants $C$ in \eqref{hhgsfhdgfhfd} and \eqref{eqci2} do not depend on the maximum eigenvalue of the conductivity $a$. Only the proof of Proposition \ref{jhgsjhgdg6354r5df} has to be adapted and the part of the proof below Proposition \ref{jhgsjhgdg6354r5df} remains unchanged. This requires an application of the elements of lemmas \ref{controlG}, \ref{sumLemma}, \ref{diffLem} and \ref{hsgsjhgsjg3e3ee} to buffer sub-domains $(\Omega_i')^r/\Omega_i'$.
 The main point is to observe that the decay of the Green's function in $(\Omega_i')^r/\Omega_i'$ can be bounded independently from $\gamma$ (due to the maximum principle).

 Observe that the choice of the sub-domain $\Omega_i$ in \eqref{loclakjededwslkajlei23} can be chosen to be the same as in \eqref{loclakjlddrfredkajlei23TT} if its intersection with high contrast inclusions is void (i.e., if the maximum eigenvalue of $a$ over $\Omega_i$ remains bounded independently from $\gamma$); otherwise the choice of $\Omega_i$ in \eqref{loclakjededwslkajlei23} has to be enlarged (when compared to that associated with \eqref{loclakjlddrfredkajlei23TT}) to contain the high-contrast inclusion (plus its buffer).

\section{The basis remains accurate for parabolic PDEs.}\label{kjsdkjhdkjshkd}

The computational gain of the method proposed in this paper is particularly significant for time-dependent problems. One such problem is the parabolic equation associated with the operator $-\diiv(a\nabla)$. More precisely, consider the time-dependent partial differential equation

\begin{equation}\label{parabolicscalarproblem0}
\begin{cases}
   \partial_t u(x,t) -\diiv \Big(a(x)  \nabla u(x,t)\Big)=g(x,t) \quad  (x,t) \in \Omega_T;\, g \in L^2(\Omega_T), \\
    u=0 \quad \text{on}\quad \partial \Omega_T,
    \end{cases}
\end{equation}
where $a$ and $\Omega$ satisfy the same assumptions as those associated with PDE \eqref{scalarproblem0}, $\Omega_T:=\Omega \times [0,T]$ for some $T>0$ and $\partial \Omega_T:=( \partial \Omega \times [0,T]) \cup (\Omega \times \{t=0\})$.

Let $V_h$ be the finite-dimensional approximation space defined in \eqref{Vh}. Let $u_h$ be the finite element solution of \eqref{parabolicscalarproblem0}, i.e., $u_h$ can be decomposed as
\begin{equation}
u_h(x,t)=\sum_i c_i(t) \psi_i(x),
\end{equation}
and solves for all $j$
\begin{equation}\label{jhgjhgjghggg}
(\psi_j, \partial_t u_h)_{L^2(\Omega)}=-a[\psi_j,u_h]+(\psi_j, g)_{L^2(\Omega)},
\end{equation}
with $a[v,w]:=\int_\Omega (\nabla v)^T a \nabla w$. Write
\begin{equation}
\|v\|_{L^2(0,T,H^1_0(\Omega))}^2:=\int_0^T \int_{\Omega} |\nabla v|^2(x,t)\,dx\,dt.
\end{equation}

\begin{Theorem}\label{kshjhgshdhg}
We have
\begin{equation}
 \big\|(u-u_h)(.,T)\big\|_{L^2(\Omega)}+\|u-u_h\|_{L^2(0,T,H^1_0(\Omega))}\leq C \|g\|_{L^2(\Omega_T)} (h+h^{2-2\alpha}).
\end{equation}
\end{Theorem}
\begin{proof}
The proof is a generalization of the proof found in \cite{MR2377253} (in which approximation spaces are constructed via harmonic coordinates).
Let $\mathcal{A}_T$ be the bilinear form on $L^2(0,T,H^1_0(\Omega))$ defined by
\begin{equation}\label{jhgjahgsjdg}
\mathcal{A}_T[w_1,w_2]:=\int_0^T a[w_1,w_2]\,dt.
\end{equation}
Observe that for all $v \in L^2(0,T,V_h)$,
\begin{equation}\label{hjhgwhgwguygeye}
 \big(v, \partial_t (u-u_h)\big)_{L^2(\Omega_T)}+\A_T[v,u-u_h]=0.
\end{equation}
Writing $\A_T[v]:=\A_T[v,v]$, we deduce that for $v \in L^2(0,T,V_h)$,
\begin{equation}\label{hjhgwhgwguygsdeeye}
\begin{split}
\frac{1}{2}&\big\|(u-u_h)(.,T)\big\|_{L^2(\Omega)}^2+\A_T[u-u_h]=\\& \big(u-v, \partial_t (u-u_h)\big)_{L^2(\Omega_T)}+\A_T[u-v,u-u_h].
\end{split}
\end{equation}
Using $\partial_t u_h$ in \eqref{jhgjhgjghggg} and integrating, we obtain that
\begin{equation}\label{jhgkjhkjhjhgjghggg}
\|\partial_t u_h\|_{L^2(\Omega_T)}^2+\frac{1}{2}a\big[u_h(.,T),u_h(.,T)\big]=\big(\partial_t u_h, g\big)_{L^2(\Omega_T)}.
\end{equation}
Using Minkowski's inequality, we deduce that
\begin{equation}\label{jhgkjsgsugyuyghkjhjhgjghggg}
\|\partial_t u_h\|_{L^2(\Omega_T)}^2+a\big[u_h(.,T),u_h(.,T)\big]\leq C \|g\|_{L^2(\Omega_T)}^2.
\end{equation}
Similarly,
\begin{equation}\label{jhgkjshgsugyuyghkjhjhggjghggg}
\|\partial_t u\|_{L^2(\Omega_T)}^2+a\big[u(.,T),u(.,T)\big]\leq C \|g\|_{L^2(\Omega_T)}^2.
\end{equation}
Using Cauchy-Schwartz and Minkowski inequalities in \eqref{hjhgwhgwguygsdeeye}, we obtain that
\begin{equation}\label{hjhgwhgwguygyuygytsdeeye}
\begin{split}
\big\|(u-u_h)(.,T)\big\|_{L^2(\Omega)}^2+\A_T[u-u_h]\leq  C \|u-v\|_{L^2(\Omega_T)} \|g\|_{L^2(\Omega_T)} +C \A_T[u-v].
\end{split}
\end{equation}
Take $v=\mathcal{R}_h u$ to be the projection of $u$ onto $L^2(0,T,V_h)$ with respect to the bilinear form $\A_T$.  Observing that
$-\diiv(a\nabla u)=g-\partial_t u$ with $(g-\partial_t u) \in L^2(\Omega_T)$, we obtain from Theorem \ref{thnloc} that
\begin{equation}\label{jhsgjdhgsdh3}
\big(\A_T[u-\mathcal{R}_h u]\big)^\frac{1}{2}\leq C \|g\|_{L^2(\Omega_T)} (h+h^{2-2\alpha}).
\end{equation}
Let us now show (using a standard duality argument) that
\begin{equation}\label{jhjijnsg}
\|u-\mathcal{R}_h u\|_{L^2(\Omega_T)}\leq C (h+h^{2-2\alpha})^2 \|g\|_{L^2(\Omega_T)}.
\end{equation}
Choose $v^*$ to be the solution of the following linear problem: For all $w\in L^2(0,T,H^1_0(\Omega))$
\begin{equation}\label{jhsg}
\mathcal{A}_T[w,v^*]=(w,u-\mathcal{R}_h u)_{L^2(\Omega_T)}.
\end{equation}
Taking $w=u-\mathcal{R}_h u$ in \eqref{jhsg}, we obtain that
\begin{equation}\label{jhjihjjhgjnsg}
\|u-\mathcal{R}_h u\|_{L^2(\Omega_T)}^2=\mathcal{A}_T[u-\mathcal{R}_h u,v^*-\mathcal{R}_h v^*].
\end{equation}
Hence by Cauchy Schwartz inequality and \eqref{jhsgjdhgsdh3},
\begin{equation}\label{dddez}
\|u-\mathcal{R}_h u\|_{L^2(\Omega_T)}^2\leq C (h+h^{2-2\alpha}) \|g\|_{L^2(\Omega_T)}\big(\mathcal{A}_T[v^*-\mathcal{R}_h v^*]\big)^\frac{1}{2}.
\end{equation}
Using Theorem \ref{thnloc} again, we obtain that
\begin{equation}\label{jhsgjdhgsdh3n}
\big(\A_T[v^*-\mathcal{R}_h v^*]\big)^\frac{1}{2}\leq C \|u-\mathcal{R}_h u\|_{L^2(\Omega_T)} (h+h^{2-2\alpha}).
\end{equation}
Combining \eqref{jhsgjdhgsdh3n} with \eqref{dddez} leads to \eqref{jhjijnsg}.
Combining \eqref{hjhgwhgwguygyuygytsdeeye} with $v=\mathcal{R}_h u$, \eqref{jhjijnsg} and \eqref{jhsgjdhgsdh3}  leads to
\begin{equation}\label{hjhgwhgsjhhydwguygyuygytsdeeye}
\begin{split}
\big\|(u-u_h)(.,T)\big\|_{L^2(\Omega)}^2+\A_T[u-u_h]\leq  C (h+h^{2-2\alpha})^2 \|g\|_{L^2(\Omega_T)}^2,
\end{split}
\end{equation}
which concludes the proof of Theorem \ref{kshjhgshdhg}.
\end{proof}

\paragraph{Discretization in time.}

Let $(t_n)$ be a discretization of $[0,T]$ with time-steps $|t_{n+1}-t_n|=\Delta t$.
 Write $Z_T^h$, the subspace of $L^2(0,T,V_h)$, such that
\begin{equation}
\begin{split}
Z_T^h= \left\{v\in
L^2(0,T,V_h)\,:\,v(x,t)=\sum_{i}c_i(t)\psi_i(x),\mbox{
$c_i(t)$ are constants on }(t_n,t_{n+1}]\right\}.
\end{split}
\end{equation}
\noindent Write $u_{h,\Delta t}$, the solution in $Z_T^h$ of the following system of
implicit weak formulation (such that $u_{h,\Delta t}(x,0)\equiv 0$): For each $n$ and $\psi\in V_h$,
\begin{equation}\label{ghjddwsdcszdbbsfh52}
\begin{split}
 \big(\psi,
u_{h,\Delta t}(t_{n+1})\big)_{L^2(\Omega)}=&\big(\psi,
u_{h,\Delta t}(t_{n})\big)_{L^2(\Omega)}\\&-|\Delta t| \,a\big[\psi,u_{h,\Delta t}(t_{n+1})]+ \big(\psi,\int_{t_n}^{t_{n+1}} g(t) \,dt\big)_{L^2(\Omega)}.
\end{split}
\end{equation}
Then, we have the following theorem
\begin{Theorem}\label{jskjdhdcxjdh723}
We have
\begin{equation}\label{gadzoszwscxedd3dbbfmbh52}
\begin{split}
 \big\| (u-u_{h,\Delta t})(T)\big\|_{L^2(\Omega)}+
&\|u-u_{h,\Delta t}\|_{L^2(0,T,H^1_0(\Omega))} \leq C \big(|\Delta t|+h+h^{2-2\alpha}\big)
\\&\Big(\|\partial_t
g\|_{L^2(0,T,H^{-1}(\Omega))}+\big\|g(.,0)\big\|_{L^2(\Omega)}+\|g\|_{L^2(\Omega_T)}\Big).
\end{split}
\end{equation}
\end{Theorem}
The proof of Theorem \ref{jskjdhdcxjdh723} is similar to that of Theorem 1.6 of  \cite{MR2377253} and will not be given here.
Observe that  homogenization in space allows for a discretization in time with time steps  $\mathcal{O}(h+h^{2-2\alpha})$ without compromising the accuracy of the method.

\section{The basis remains accurate for hyperbolic PDEs.}\label{sechyperbolic}

Consider the hyperbolic partial differential equation
\begin{equation}\label{huperboliccalarproblem0}
\begin{cases}
   \rho(x) \partial_{t}^2 u(x,t) -\diiv \Big(a(x)  \nabla u(x,t)\Big)=g(x,t) \quad  (x,t) \in \Omega_T;\, g \in L^2(\Omega_T), \\
    u=0 \quad \text{on}\quad  \partial \Omega_T,\\
    \partial_t u=0  \quad \text{on}\quad  \Omega \times \{t=0\},
    \end{cases}
\end{equation}
where $a$, $\Omega$, $\Omega_T$ and $\partial \Omega_T$ are defined as in Section \ref{kjsdkjhdkjshkd}. In particular, $a$ is assumed to be only uniformly elliptic and bounded ($a_{i,j}\in L^\infty(\Omega)$). We will further assume that $\rho$ is uniformly bounded from below and above ($\rho\in L^\infty(\Omega)$ and $\operatorname{essinf} \rho(x)\geq \rho_{\min}>0$). It is straightforward to extend the results presented here to nonzero boundary conditions (provided that frequencies larger than $1/h$ remain weakly excited, because the waves equation preserves energy and
homogenization schemes can not recover energies put into high frequencies, see \cite{OwZh06c}). For the sake of conciseness, we will give those results with zero boundary conditions. PDE \eqref{huperboliccalarproblem0} corresponds to acoustic wave equations in a medium with density $\rho$ and bulk modulus $a^{-1}$.

Let $V_h$ be the finite-dimensional approximation space defined in \eqref{Vh}. Let $u_h$ be the finite element solution of \eqref{huperboliccalarproblem0}, i.e., $u_h$ can be decomposed as
\begin{equation}
u_h(x,t)=\sum_i c_i(t) \psi_i(x),
\end{equation}
and solves for all $j$
\begin{equation}\label{jhgjhgjghggsg}
(\psi_j, \partial_t^2 u_h)_{L^2(\rho,\Omega)}=-a[\psi_j,u_h]+(\psi_j, g)_{L^2(\Omega)},
\end{equation}
where
\begin{equation}\label{jhgjhgdddjghggsg}
(v, w)_{L^2(\rho,\Omega)}:=\int_{\Omega} v\,w\,\rho.
\end{equation}

\begin{Theorem}\label{kshejcfchgshdhg}
If  $\partial_t g\in L^2(\Omega_T)$ and $ g(x,0)\in L^2(\Omega)$, then
\begin{equation}
\begin{split}
 \big\|\partial_t (u-u_h)(.,T)\big\|_{L^2(\Omega)}&+\big\|u-u_h\big\|_{L^2(0,T,H^1_0(\Omega))}\leq\\& C \big(\|\partial_t g\|_{L^2(\Omega_T)}+\| g(x,0)\|_{L^2(\Omega)}\big) (h+h^{2-2\alpha}).
 \end{split}
\end{equation}
\end{Theorem}
\begin{Remark}
We refer to \cite{SymVdo09} for an analysis of the sub-optimal rate of convergence
 associated with finite-difference simulation of wave propagation in discontinuous media (see also \cite{Brown84,SeiSymes94}).
We refer to \cite{OwZh06c} for an alternative upscaling strategy based on harmonic coordinates. If the medium is locally ergodic with long range correlations \cite{BalJing10} and also characterized by scale separation then we refer to HMM based methods \cite{EnHoRu10, AbGro10}.
Homogenization based methods require that frequencies larger than $1/h$ remain weakly excited. For high frequencies, and smooth media (or away from local resonances, e.g. local, nearly resonant cavities), we refer to the sweeping pre-conditioner method
\cite{EngYing10a, EngYing10b}.
\end{Remark}

\begin{proof}
Let $\mathcal{A}_T$ be the bilinear form on $L^2(0,T,H^1_0(\Omega))$ defined in \eqref{jhgjahgsjdg}.
Observe that for all $v \in L^2(0,T,V_h)$,
\begin{equation}\label{hjhkajsbxgwhsgwguygeye}
 \big(v, \partial_t^2 (u-u_h)\big)_{L^2(\rho,\Omega_T)}+\A_T[v,u-u_h]=0.
\end{equation}
Taking $\partial_t u-\partial_t u_h-(\partial_t u- \partial_t v)$ as a test function in \eqref{hjhkajsbxgwhsgwguygeye} and integrating in time,
we deduce that for $\partial_t v \in L^2(0,T,V_h)$,
\begin{equation}\label{hjlkjkhgdewhgwguygsdeeye}
\begin{split}
\frac{1}{2}&\big\|\partial_t (u-u_h)(.,T)\big\|_{L^2(\rho,\Omega)}^2+\frac{1}{2} a\big[(u-u_h)(.,T)\big]=\\& \big(\partial_t (u-v), \partial_t^2 (u-u_h)\big)_{L^2(\rho,\Omega_T)}+\A_T[\partial_t (u-v),u-u_h],
\end{split}
\end{equation}
where $(v,w)_{L^2(\rho,\Omega_T)}:=\int_0^T \int_\Omega v\,w\,\rho\,dx\,dt$.
Taking  the derivative of the hyperbolic equation
for $u$ in time, we obtain that
\begin{equation}\label{ksjhhdksjhd}
\partial_t^3 u-\diiv(a\nabla \partial_t u)=\partial_t g.
\end{equation}
Integrating \eqref{ksjhhdksjhd} against the test function $\partial_t^2 u$ and
 observing that $\partial_t^2 u(x,0)= g(x,0)$, we also obtain that
\begin{equation}\label{jhgkgjuUiyghhukjhjhggjghggg}
\big\|\partial_t^2 u (.,T)\big\|_{L^2(\rho,\Omega)}^2+a\big[\partial_t u(.,T)\big]\leq C \big(\|\partial_t g\|_{L^2(\Omega_T)}^2+\| g(x,0)\|_{L^2(\Omega)}^2\big).
\end{equation}
Similarly, we obtain that
\begin{equation}\label{jhgkgjruyUiyghhukjhjhggjghggg}
\big\|\partial_t^2 u_h (.,T)\big\|_{L^2(\rho,\Omega)}^2+a\big[\partial_t u_h(.,T)\big]\leq C \big(\|\partial_t g\|_{L^2(\Omega_T)}^2+\| g(x,0)\|_{L^2(\Omega)}^2\big).
\end{equation}

Take $\partial_t v=\mathcal{R}_h \partial_t u$ to be the projection of $\partial_t u$ onto $L^2(0,T,V_h)$ with respect to the bilinear form $\A_T$.
Observing that
$-\diiv(a\nabla \partial_t u)=\partial_t g-\partial_t^2 u$ with $(g-\partial_t^2 u) \in L^2(\Omega_T)$, we obtain from \eqref{jhgkgjuUiyghhukjhjhggjghggg} and Theorem \ref{thnloc} that
\begin{equation}\label{jhsgjshhdgyydhgsdh3}
\big(\A_T[u-\mathcal{R}_h u]\big)^\frac{1}{2}\leq C \big(\|\partial_t g\|_{L^2(\Omega_T)}+\| g(x,0)\|_{L^2(\Omega)}\big) (h+h^{2-2\alpha}).
\end{equation}
Furthermore, using the same duality argument as in the parabolic case, we obtain that
\begin{equation}\label{ddjhgjgjjdez}
\|u-\mathcal{R}_h u\|_{L^2(\rho,\Omega_T)}\leq C (h+h^{2-2\alpha})^2 \big(\|\partial_t g\|_{L^2(\Omega_T)}+\|g(x,0)\|_{L^2(\Omega)}\big).
\end{equation}
Using Cauchy-Schwartz and Minkowski inequalities  and the above estimates in  \eqref{hjlkjkhgdewhgwguygsdeeye}, we obtain that
\begin{equation}\label{hjlkjkhgdewhgwguygsdeejjye}
\begin{split}
&\big\|\partial_t (u-u_h)(.,T)\big\|_{L^2(\rho,\Omega)}^2+ a\big[(u-u_h)(.,T)\big]\leq\\& C (h+h^{2-2\alpha}) \big(\A_T[u-u_h]+
\|\partial_t g\|_{L^2(\Omega_T)}+\|g(x,0)\|_{L^2(\Omega)}\big).
\end{split}
\end{equation}
We conclude using Grownwall's lemma.
\end{proof}

\begin{table}[!tbp]
\begin{center}
\begin{tabular}{|c||c|c|c|}
\hline $h$ &  $L^2$ & $H^1$ & $L^{\infty}$
\tabularnewline \hline \hline
\begin{tabular}{c}
0.5\tabularnewline \hline 0.25\tabularnewline \hline
0.125\tabularnewline \hline 0.0625\tabularnewline
\end{tabular}&
\begin{tabular}{c}
0.0119\tabularnewline \hline 0.0057\tabularnewline \hline
0.0027\tabularnewline \hline 0.0005\tabularnewline
\end{tabular}&
\begin{tabular}{c}
0.0913\tabularnewline \hline 0.0664\tabularnewline \hline
0.0482\tabularnewline \hline 0.0207\tabularnewline
\end{tabular}&
\begin{tabular}{c}
0.0157\tabularnewline \hline 0.0115\tabularnewline \hline
0.0075\tabularnewline \hline 0.0032\tabularnewline
\end{tabular}
\tabularnewline \hline
\end{tabular}
\caption{Example 1 of Section 3 of \cite{MR2292954} (trigonometric multi-scale, see also \cite{MiYu06}) with
$\alpha=1/2$. \label{tab:example1-1}}
\end{center}
\end{table}

\section{Numerical experiments.}\label{numexperiments}

\begin{figure}[httb]
  \begin{center}
    \subfigure[$L^2$ error]
    {\includegraphics[width=0.35\textwidth,height= 0.3\textwidth]{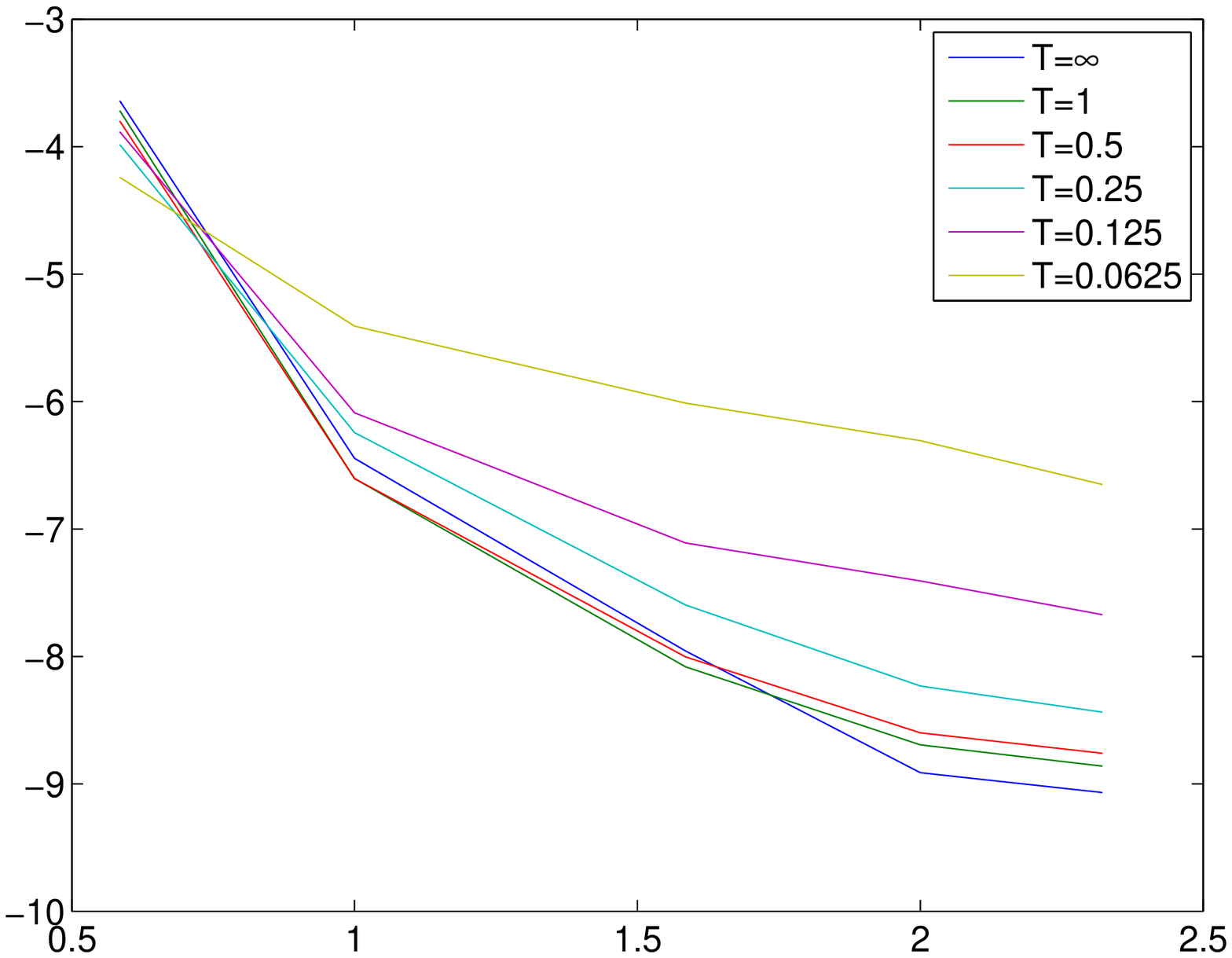}}
    \goodgap
    \subfigure[$H^1$ error]
    {\includegraphics[width=0.35\textwidth,height= 0.3\textwidth]{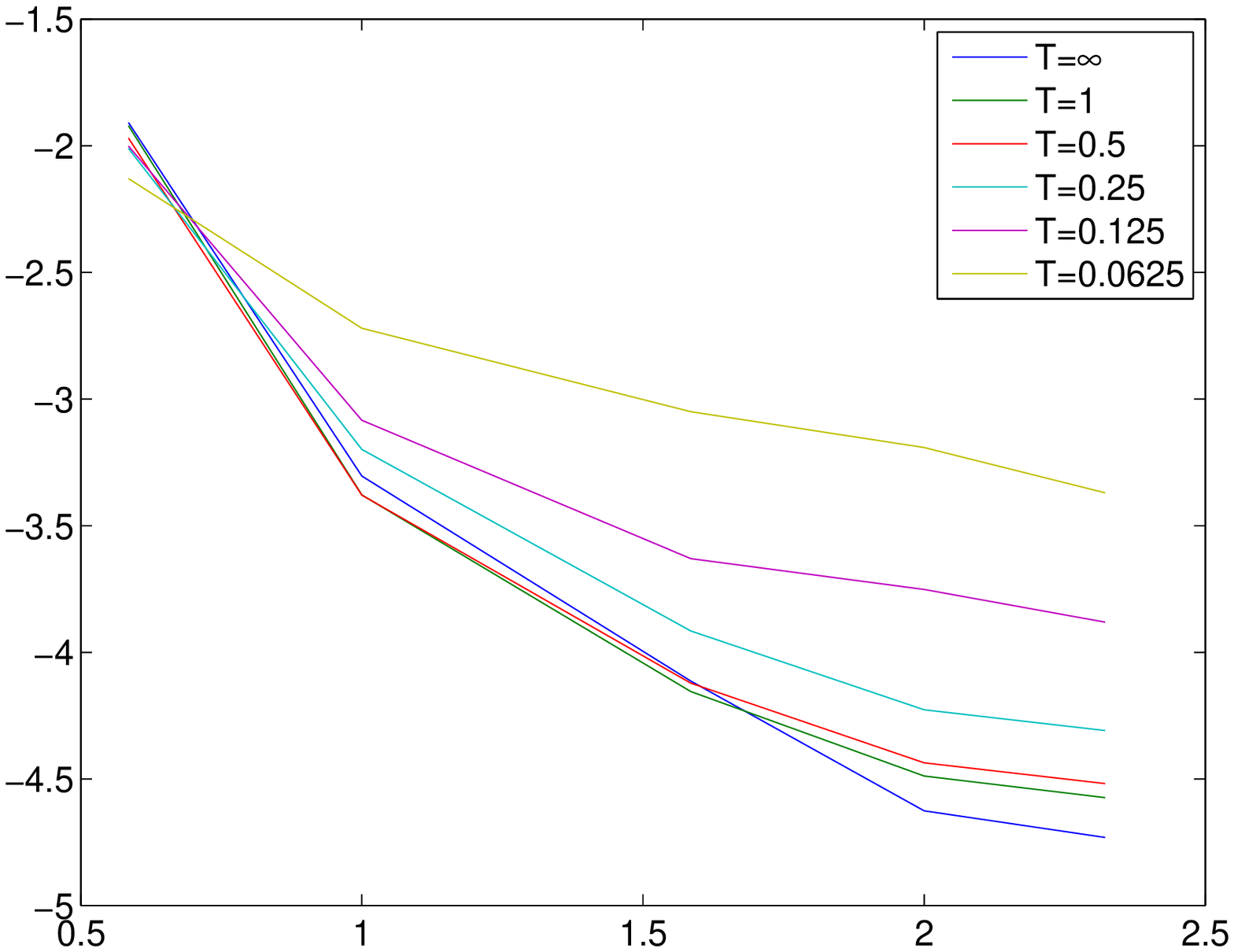}}
    \caption{Example 5 of Section 3 of \cite{MR2292954} (percolation at criticality). Logarithm (in base $2$) of the error with respect to $\log_2(h_0/h)$ (for $h=0.125$)  and the value of $T$ used in \eqref{loclakjlkajlei23}.}
    \label{fig:example4-2}
  \end{center}
\end{figure}
\begin{figure}[httb]
  \begin{center}
    \subfigure[$L^2$ error]
    {\includegraphics[width=0.35\textwidth,height= 0.3\textwidth]{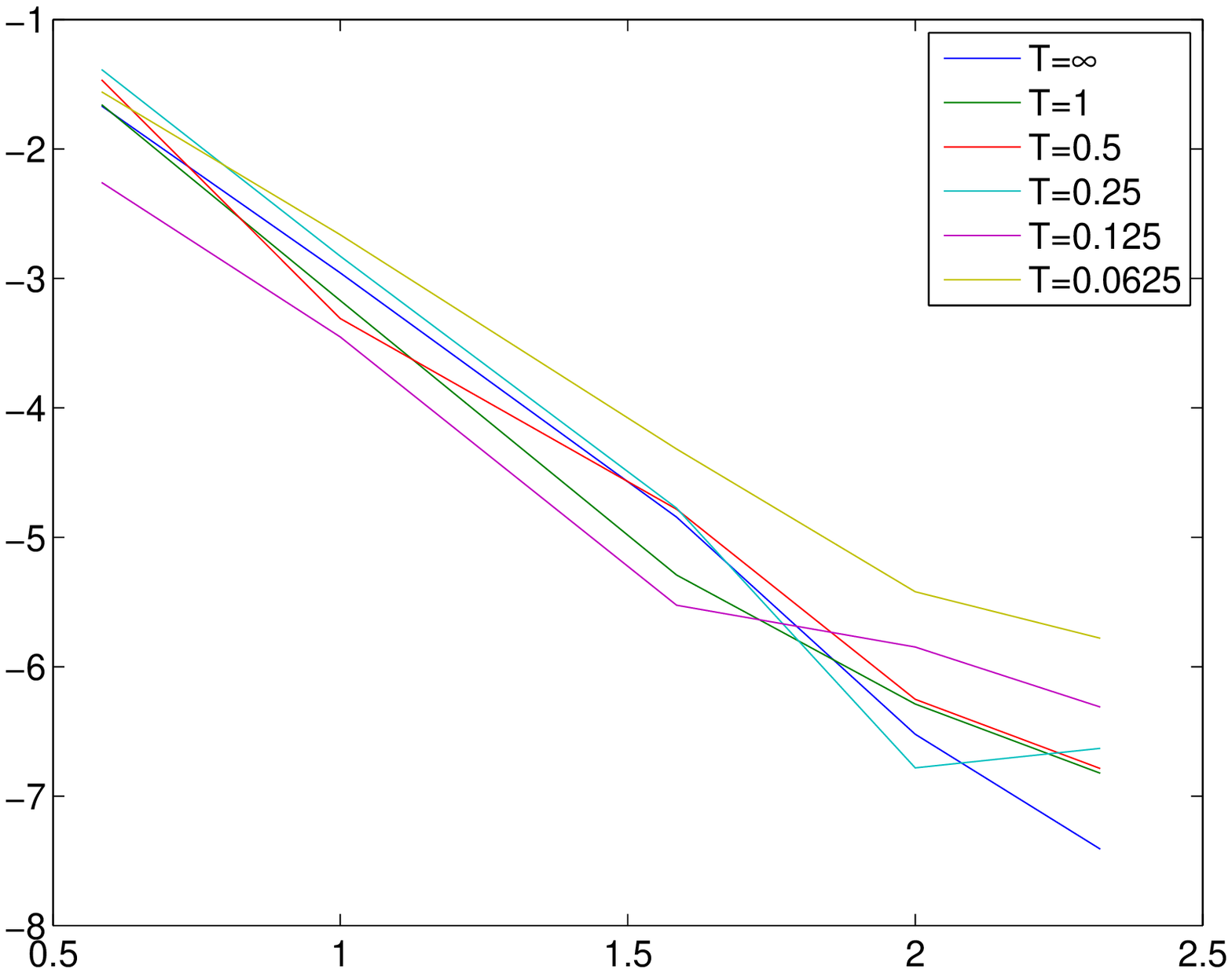}}
    \goodgap
    \subfigure[$H^1$ error]
    {\includegraphics[width=0.35\textwidth,height= 0.3\textwidth]{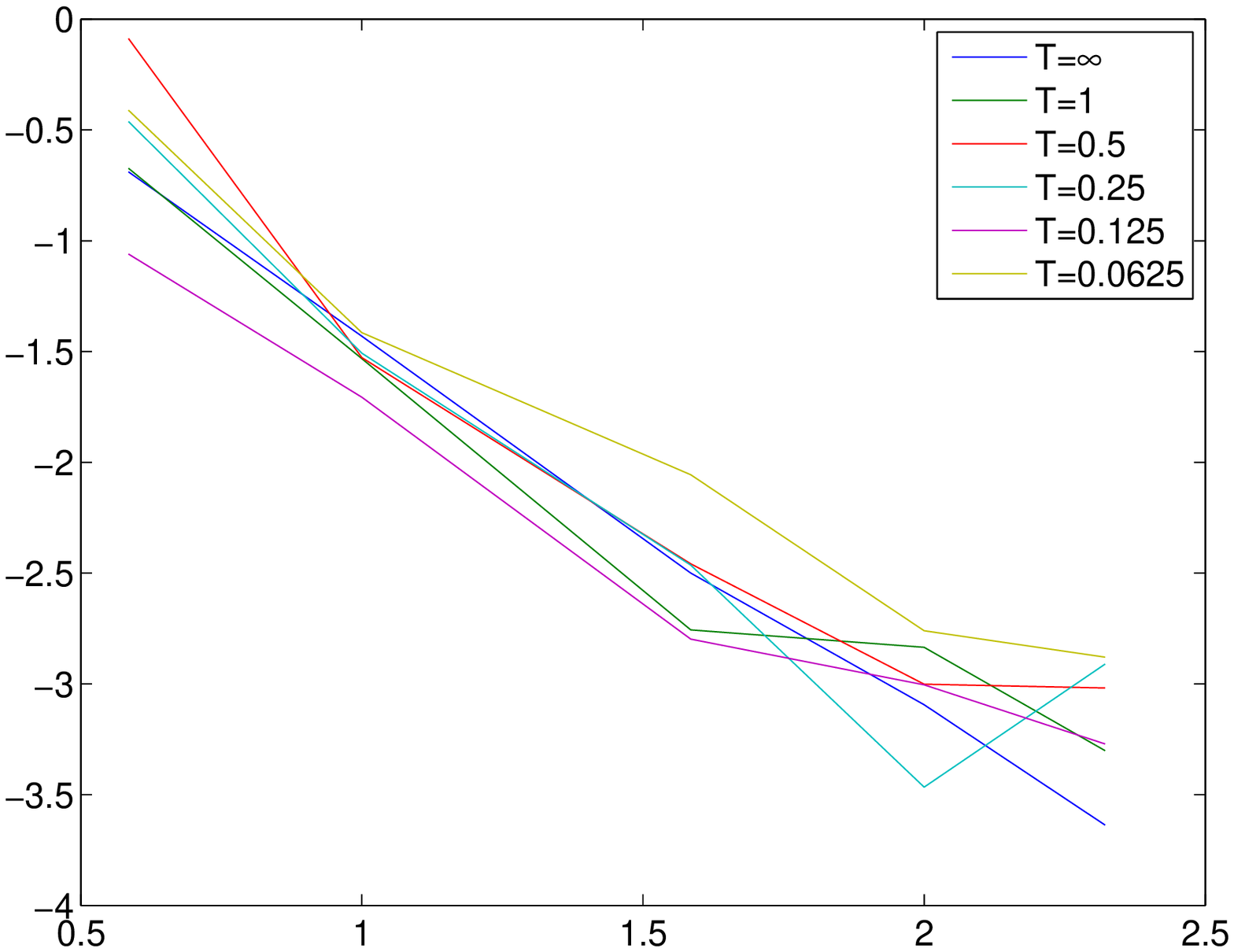}}
    \caption{Example 3 of Section 3 of \cite{MR2292954} (exponential of a sum of trigonometric functions with strongly overlapping frequencies). Logarithm (in base $2$) of the error with respect to $\log_2(h_0/h)$ (for $h=0.125$)  and the value of $T$ used in \eqref{loclakjlkajlei23}.}
    \label{fig:example3-2}
  \end{center}
\end{figure}

\begin{figure}[!tbp]
  \begin{center}
    \includegraphics[width=0.35\textwidth,height= 0.3\textwidth]{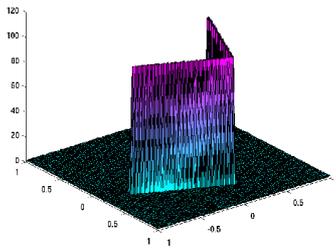}
    \caption{High conductivity channel.}\label{exa:channel}
\end{center}
\end{figure}

\begin{figure}[httb]
  \begin{center}
    \subfigure[$L^2$ error]
    {\includegraphics[width=0.35\textwidth,height= 0.3\textwidth]{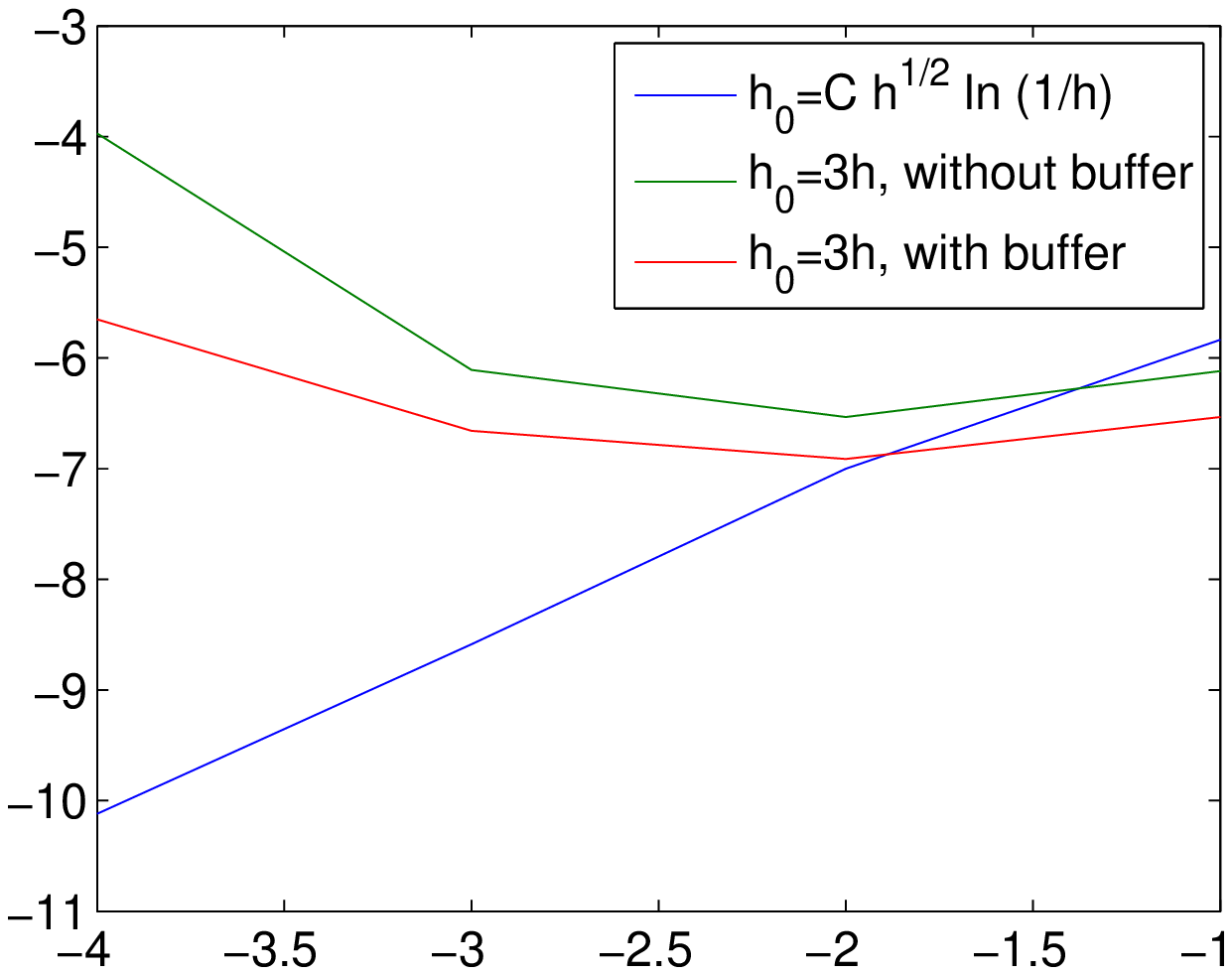}}
    \goodgap
    \subfigure[$H^1$ error]
    {\includegraphics[width=0.35\textwidth,height= 0.3\textwidth]{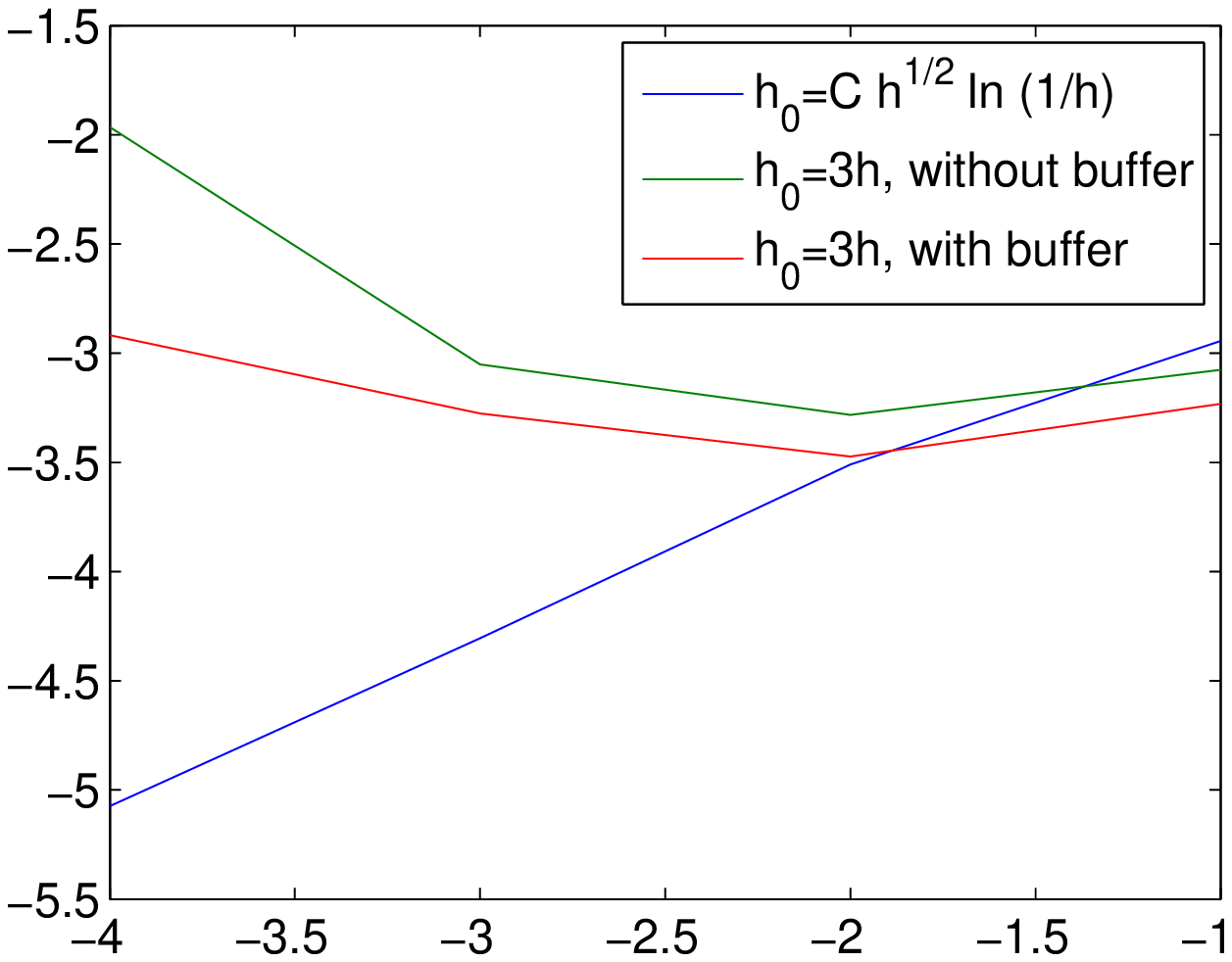}}
    \caption{High conductivity channel (Figure \ref{exa:channel}). The $x$-axis shows $\log_2(h)$,
    the $y$-axis shows the $\log_2$ of the error in $L^2$ and $H^1$-norm.  The three cases for the localization are $h_0=\mathcal{O}(\sqrt{h}\ln \frac{1}{h})$ with a buffer around the high conductivity channel (see Sub-section \ref{highcontrast}) of size $\mathcal{O}(\sqrt{h}\ln \frac{1}{h})$, $h_0=3h$ with no buffer around the high conductivity channel and $h_0=3h$ with a buffer around the high conductivity channel of size $3h$.}
    \label{fig:example2-3}
  \end{center}
\end{figure}

\begin{figure}[httb]
  \begin{center}
    \subfigure[$u$]
    {\includegraphics[width=0.35\textwidth,height= 0.3\textwidth]{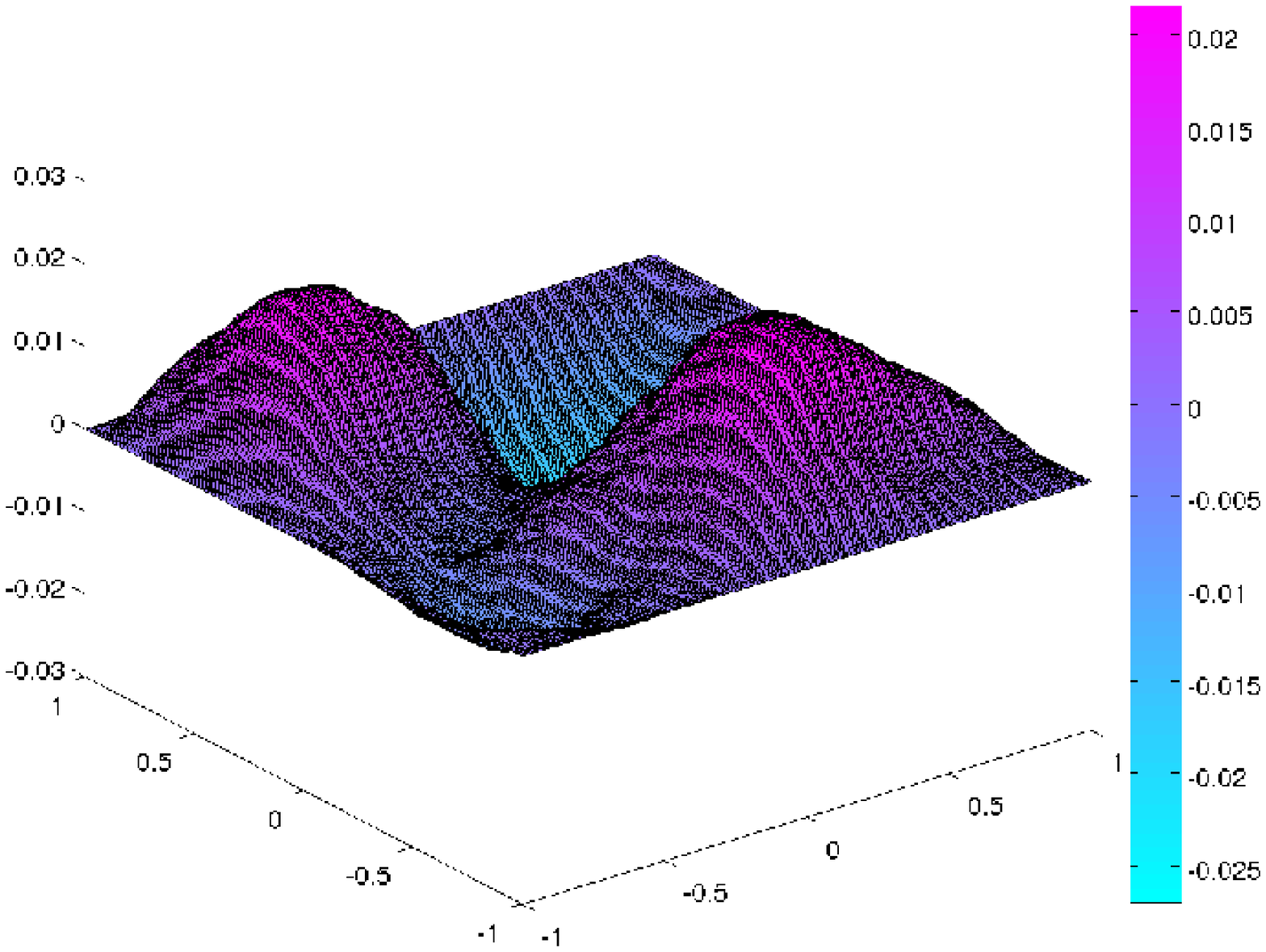}}
    \goodgap
    \subfigure[$u_h$]
    {\includegraphics[width=0.35\textwidth,height= 0.3\textwidth]{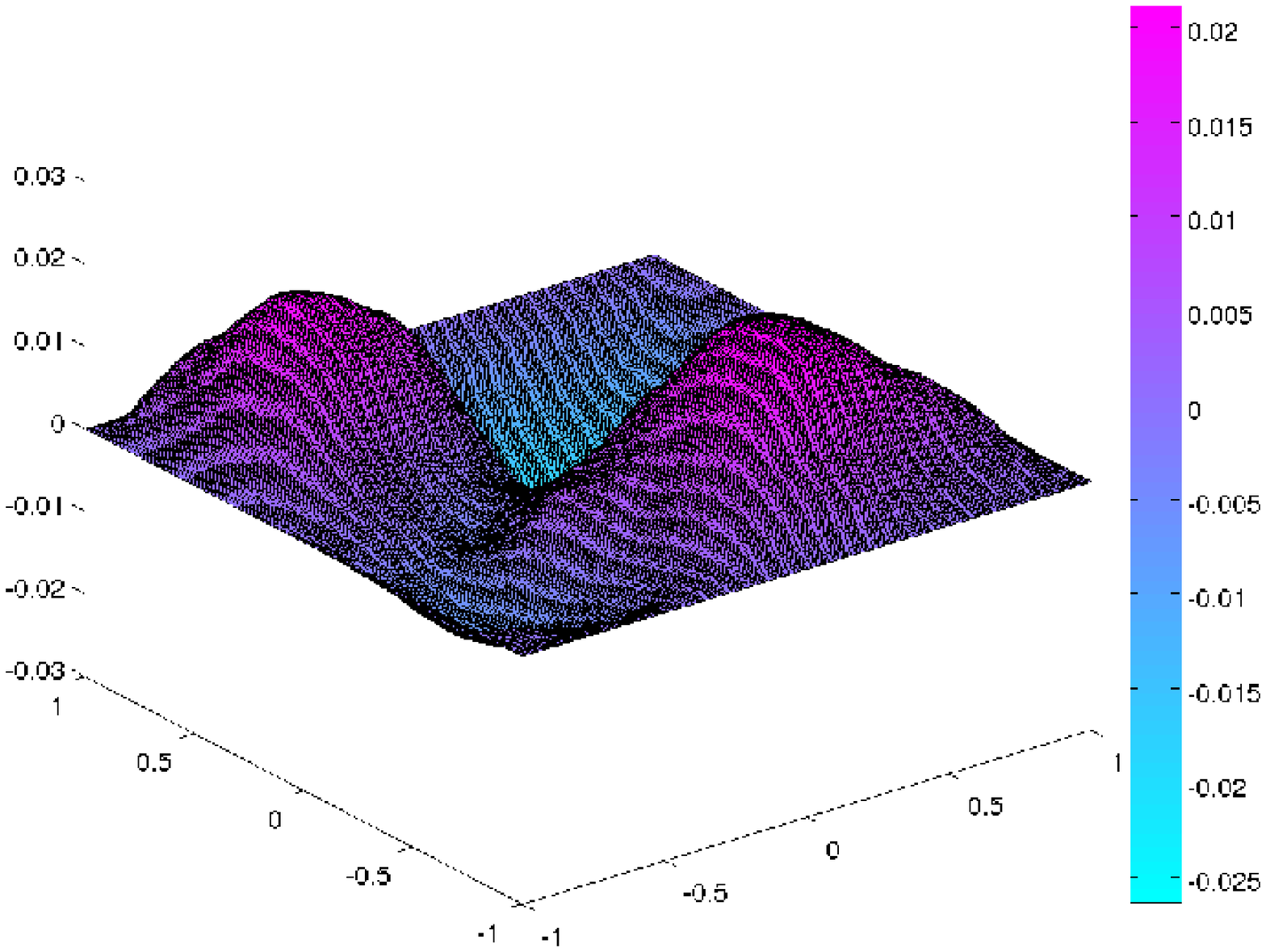}}
    \caption{Wave equation. Trigonometric case, fine mesh solution, $h=0.125$, $h_0=3h$, $T=h$. The $L^2$, $H^1$ and $L^\infty$ relative numerical errors are $0.0339$, $0.1760$ and $0.0235$.}
    \label{fig:wave2}
  \end{center}
\end{figure}

\begin{figure}[httb]
  \begin{center}
    \subfigure[$u$]
    {\includegraphics[width=0.35\textwidth,height= 0.3\textwidth]{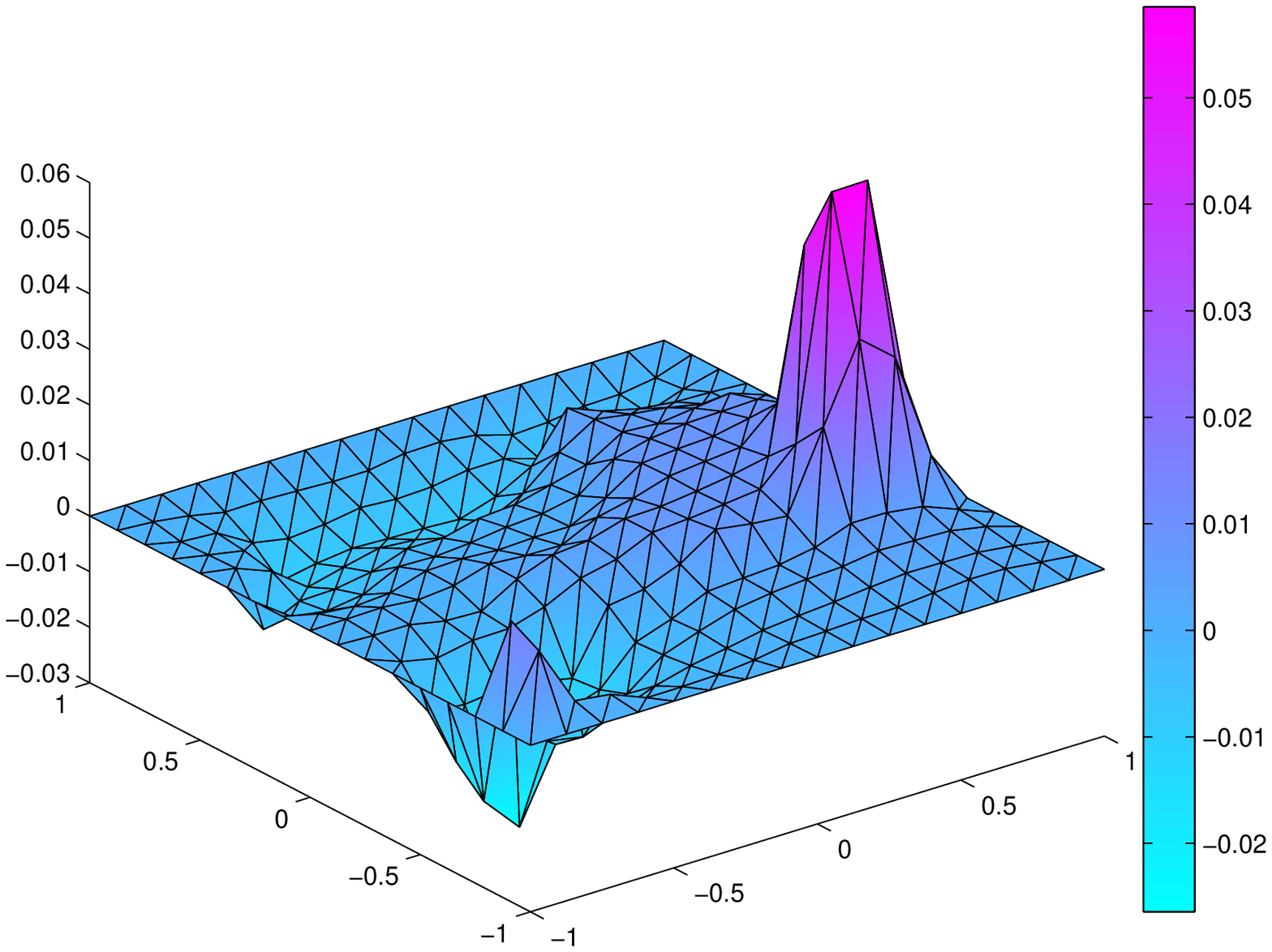}}
    \goodgap
    \subfigure[$u_h$]
    {\includegraphics[width=0.35\textwidth,height= 0.3\textwidth]{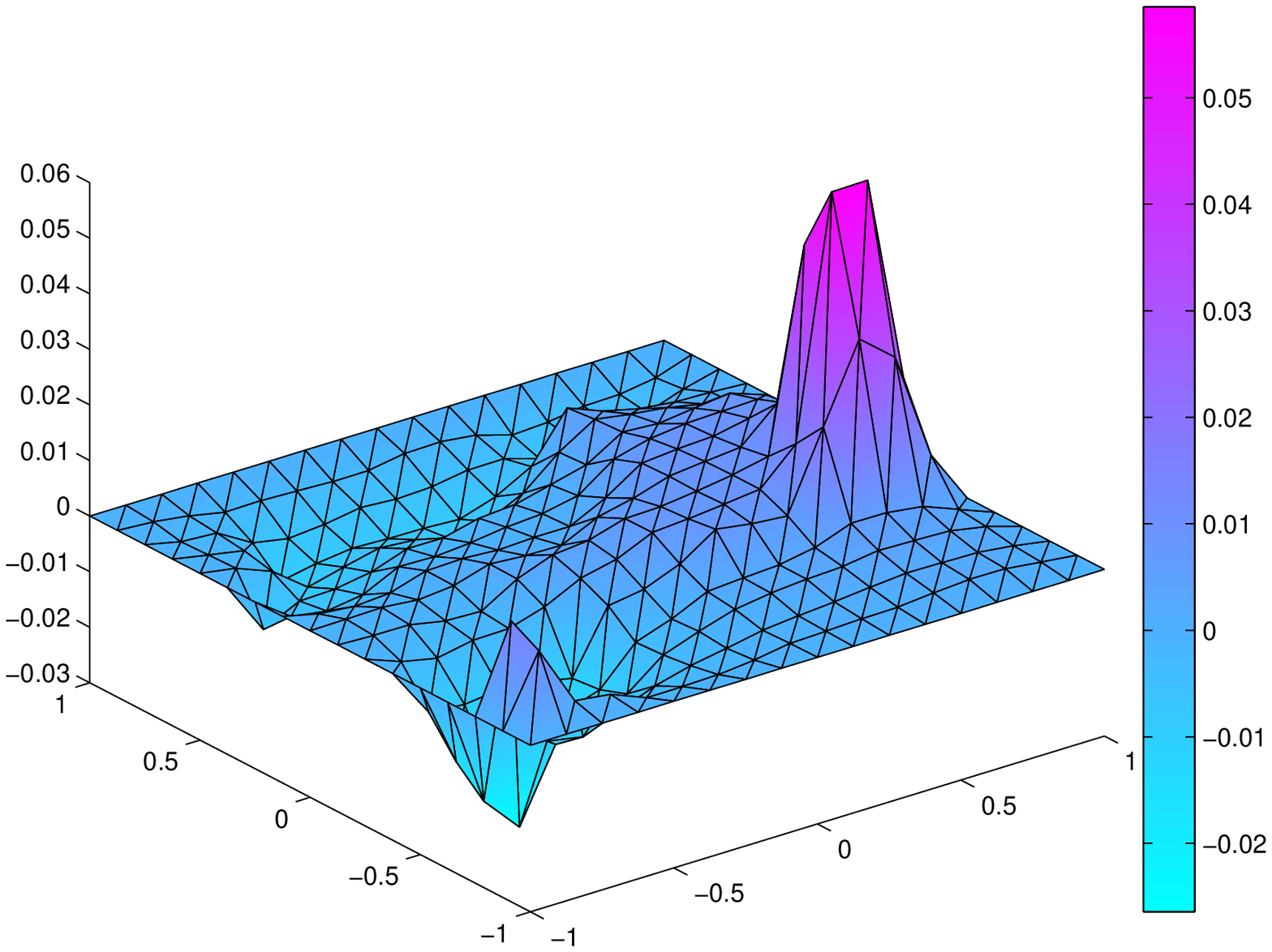}}
    \caption{Wave equation. Channel case, coarse mesh solution, $h=0.125$, $h_0=3h$, $T=h$. The $L^2$, $H^1$ and $L^\infty$ relative numerical errors are $0.0439$, $0.2684$ and $0.0389$.}
    \label{fig:wave3}
  \end{center}
\end{figure}

\subsection{Elliptic equation.}
We compute the solutions of \eqref{scalarproblem0} up to time $1$ on the fine mesh and in  the finite-dimensional approximation space $V_h$ defined in \eqref{Vh}.
The physical domain is the square $[-1,1]^2$. Global equations are
solved on a fine triangulation with $66049$ nodes and $131072$
triangles.

The elements $(\varphi_i)$ of Sub-section \ref{jhsgdjhdghgdhed} are weighted extended B-splines (WEB) \cite{Splines1,Splines2} (obtained by tensorizing one-dimensional elements and using weight function $(1-x^2)(1-y^2)$ to enforce the Dirichlet boundary condition).
The order of accuracy is not affected by
the choice of weight function given that the boundary is piecewise smooth.  Our motivation for using WEB elements lies in the fact that, with those elements, (Dirichlet) boundary conditions become simple to enforce. This being said, any finite elements satisfying the properties \eqref{Xprop1}, \eqref{Xprop2}, \eqref{Controlphi} and  \eqref{Controlphi2sum} would be adequate \cite{Brae07}.

We write $h$ the size of the coarse mesh.
Elements $\psi_i$ are obtained by solving \eqref{loclakjlkajlei23} on localized sub-domains of size $h_0$.
Table \ref{tab:example1-1} shows errors with
$\alpha=1/2$ for $a$ given by \eqref{ex1} (Example 1 of Section 3 of \cite{MR2292954}, trigonometric multi-scale, see also \cite{MiYu06}), i.e., for
\begin{equation}\label{ex1}
\begin{split}
a(x):=&\frac{1}{6}(\frac{1.1+\sin(2\pi x/\epsilon_{1})}{1.1+\sin(2\pi
y/\epsilon_{1})}+\frac{1.1+\sin(2\pi y/\epsilon_{2})}{1.1+\cos(2\pi
x/\epsilon_{2})}+\frac{1.1+\cos(2\pi x/\epsilon_{3})}{1.1+\sin(2\pi
y/\epsilon_{3})}+\\
&\frac{1.1+\sin(2\pi y/\epsilon_{4})}{1.1+\cos(2\pi
x/\epsilon_{4})}+\frac{1.1+\cos(2\pi x/\epsilon_{5})}{1.1+\sin(2\pi
y/\epsilon_{5})}+\sin(4x^{2}y^{2})+1),
\end{split}
\end{equation}
where $\epsilon_{1}=\frac{1}{5}$,$\epsilon_{2}=\frac{1}{13}$,$\epsilon_{3}=
\frac{1}{17}$,$\epsilon_{4}=\frac{1}{31}$,$\epsilon_{5}=\frac{1}{65}$.

Figure \ref{fig:example4-2} shows the logarithm (in base $2$) of the error with respect to $\log_2(h_0/h)$ (for $h=0.125$) and the value of $T$ used in \eqref{loclakjlkajlei23} for $a$ given by
Example 5 of Section 3 of \cite{MR2292954} (percolation at criticality, the conductivity of each site is equal to $\gamma$ or $1/\gamma$
with probability $1/2$ and $\gamma=4$).

Figure \ref{fig:example3-2} shows the logarithm (in base $2$) of the error with respect to $\log_2(h_0/h)$ (for $h=0.125$) and the value of $T$ used in \eqref{loclakjlkajlei23} for $a$ given by
Example 3 of Section 3 of \cite{MR2292954}, i.e.,  $a(x)=e^{h(x)}$, with
$
h(x)=\sum_{|k|\leq R}(a_{k}\sin(2\pi k\cdot x)+b_{k}\cos(2\pi k\cdot x))
$,
where $a_{k}$ and $b_{k}$ are independent uniformly distributed random
variables on $[-0.3,0.3]$ and $R=6$.

\begin{Remark}
Two factors contribute to the error plots shown in  figures \ref{fig:example4-2} and \ref{fig:example3-2}: a localization error which becomes dominant when $h_0/h$ is small (i.e. the fact \eqref{lakjlkajlei23} is not solved over the whole domain $\Omega$) and the distortion of the transfer property resulting from the $1/T$ term in \eqref{loclakjlkajlei23TT}. As expected both figures show that when $h_0/h$ is large, the error due to the distortion of the transfer property is dominant and is minimized by a large $T$. However, when $h_0/h$ is small, the localization error is dominant and  is minimized by a small $T$. The fact that in Figure \ref{fig:example3-2} this error is minimized by the second smallest $T$ instead of the smallest $T$ is explained by the fact that the localization error remains  bounded when $h_0/h$ is of the order of one whereas the error due to the distortion of the transfer property blows up as $T\downarrow 0$.
The fact that in both figures, curves associated with different $T$ interest each other, is indicative of the fact that for intermediate values of $h_0/h$, the error can be minimized via a fine-tuning of $T$ as explained in section \ref{jshgjhgdshjghdjhe}. The differences in the locations of these intersections can be explained by a larger localization error associated with
 the example of Figure \ref{fig:example3-2} (due to longer correlation ranges). In particular,
 the comparison between figures \ref{fig:example4-2} and \ref{fig:example3-2} indicates larger errors for Figure \ref{fig:example3-2}.
\end{Remark}

\subsection{High contrast, with and without buffer.}
In this example, $a$ is characterized by a fine and long-ranged high
conductivity channel (Figure \ref{exa:channel}). We choose $a(x)=100$, if $x$ is in the channel,
and $a(x)$ is the percolation medium, if $x$ is not in the channel (the conductivity of each site, not in channel, is equal to $\gamma$ or $1/\gamma$
with probability $1/2$ and $\gamma=4$). Figure \ref{fig:example2-3} shows the $log_2$ of the numerical error (in $L^2$ and $H^1$ norm)
versus $log_2(h)$.
The three cases for the localization are $h_0=\mathcal{O}(\sqrt{h}\ln \frac{1}{h})$ with a buffer $b_i$ around the high conductivity
channel (see Sub-section \ref{highcontrast}) of size $\mathcal{O}(\sqrt{h}\ln \frac{1}{h})$, $h_0=3h$ with no buffer around the high conductivity channel and $h_0=3h$ with a buffer $b_i$ around the high conductivity channel of size $3h$.
The first case shows that the method of Sub-section \ref{highcontrast} is converging as expected. The second case shows that, as expected, taking $\alpha=1$, does not guarantee convergence. The third case shows that adding a buffer around the high conductivity channel  improves  numerical errors but is not sufficient to guarantee convergence (as expected,  we also need $\alpha<1$). The percolating background medium has been re-sampled for each case; the effect of this re-sampling can be seen for the largest value of $h$ (i.e., $\log_2(h)=-1$).

\subsection{Wave equation.}
We compute the solutions of \eqref{huperboliccalarproblem0} up to time $1$ on the fine mesh and in  the finite-dimensional approximation space $V_h$ defined in \eqref{Vh}.
 The initial condition is
$u(x,0)=0$ and $u_t(x,0)=0$. The boundary condition is $u(x,t)=0$, for
$x\in\partial\Omega$. The density is uniformly equal to one and we choose $g=\sin(\pi x)\sin(\pi y)$.
 Figure \ref{fig:wave2} shows the fine mesh solutions $u$ and $u_h$ at time one,
for $a$ given by the trigonometric example \eqref{ex1}, with  $h=0.125$, $h_0=3h$ and $T=h$. Figure \ref{fig:wave2} shows the fine mesh solutions $u$ and $u_h$ at time one,
for $a$ given by the high conductivity channel example (Figure \ref{exa:channel}), with $h=0.125$, $h_0=3h$ and $T=h$.

We refer to \cite{OwZh11movies} for a list of movies on  the numerical homogenization of the wave equation with and without high contrast and with and without buffers (extended buffers in the high contrast case).

\appendix
\section{Proof of Proposition \ref{jhgsjhgdg6354r5df}.}\label{jhgsjhgd33ge3e}

The proof of Proposition \ref{jhgsjhgdg6354r5df} is a generalization of the proof of the control of the resonance error in periodic medium given in \cite{Gloria10}.

First we need the following lemma, which is the cornerstone of Cacciopoli's inequality.
\begin{Lemma}\label{LemCaccio}
Let $D$ be a sub-domain of $\Omega$ with piecewise Lipschitz boundary, and let $v$ solve
\begin{equation}\label{scaderblem0Cac}
\begin{cases}
    \frac{v}{T}-\diiv \Big(a(x)  \nabla v(x)\Big)=f(x) \quad  x \in D; f \in H^{-1}(D),\\
    v=0 \quad \text{on}\quad \partial D,
    \end{cases}
\end{equation}
Let $\zeta \,:\, D \rightarrow \R^+$ be a function of class $C^1$ such that $\zeta$ is identically null on an open neighborhood of the support
of $f$. Then,
\begin{equation}\label{scsddeac}
\int_{D} \big|\nabla (\zeta v)\big|^2\leq C \int_{D} v^2 |\nabla \zeta|^2,
\end{equation}
where $C$ only depends on the essential supremum and infimum of the maximum and minimum eigenvalues of $a$ over $D$.
\end{Lemma}
\begin{proof}
Multiplying \eqref{scaderblem0Cac} by $\zeta^2 v$ and integrating by parts, we obtain that
\begin{equation}
\int_{D}\zeta \frac{v^2}{T}+\int_D \nabla (\zeta^2 v) a \nabla v=0.
\end{equation}
Hence,
\begin{equation}
\int_{D}\zeta \frac{v^2}{T}+\int_D \nabla (\zeta v) a \nabla (\zeta v)= \int_D v^2 \nabla \zeta a \nabla \zeta,
\end{equation}
which concludes the proof.
\end{proof}

\begin{Lemma}\label{controlG}
Let $D$ be a sub-domain of $\Omega$ with piecewise Lipschitz boundary. Write $G_{T,D}$ the Green's function of the operator $\frac{1}{T}-\diiv(a\nabla)$ with Dirichlet boundary condition on $\partial D$. Then,
\begin{equation}
G_{T,D}(x,y)\leq \frac{C}{|x-y|^{d-2}} \exp\big(-\frac{|x-y|}{C\sqrt{T}}\big),
\end{equation}
where $C$ only depends on $d$ and the essential supremum and infimum of the maximum and minimum eigenvalues of $a$ over $D$.
\end{Lemma}
\begin{proof}
Extending $a$ to $\R^d$ and using the maximum principle, we obtain that
\begin{equation}
G_{T,D}(x,y)\leq G_{T,\R^d}(x,y),
\end{equation}
we conclude by using the exponential decay of the Green's function in $\R^d$ (we refer to Lemma 2 of \cite{Gloria10}).
\end{proof}

\begin{Lemma}\label{sumLemma}
Let $\psi_{i,T}$ be the solution of \eqref{loclakjlkajlei23TT} and $\psi_{i,T,\Omega_i}$ the solution of \eqref{loclakjlddrfredkajlei23TT}.
Let $\Omega_i'$ be a sub-domain of $\Omega_i$ such that $S_i \subset \Omega_i'$ and $\dist(S_i, \Omega_i/\Omega_i')>0$. We have
\begin{equation}\label{contpsiext1}
\big\|\psi_{i,T}\big\|_{H^1(\Omega/\Omega_i')} \leq \frac{C h^{\frac{d}{2}-1}}{\big(\dist(S_i,\Omega/\Omega_i')\big)^{d}} \exp\Big(-\frac{\dist(S_i,\Omega/\Omega_i')}{C \sqrt{T}}\Big),
\end{equation}
and
\begin{equation}\label{contpsiext2}
\big\|\psi_{i,T,\Omega_i}\big\|_{H^1(\Omega_i/\Omega_i')} \leq \frac{C h^{\frac{d}{2}-1}}{\big(\dist(S_i,\Omega/\Omega_i')\big)^{d}} \exp\Big(-\frac{\dist(S_i,\Omega/\Omega_i')}{C \sqrt{T}}\Big).
\end{equation}
\end{Lemma}

\begin{proof}
For $A\subset \Omega$, write $A^r$ the set of points of $\Omega$ that are at distance at most $r$ from $A$.
Let us now use Cacciopoli's inequality to bound $\int_{\Omega/\Omega_i'} |\nabla \psi_{i,T}|^2$.  Using Lemma \ref{LemCaccio} with  $\zeta$ identically equal to one on $\Omega/\Omega_i'$, zero on $(\Omega/\Omega_i')^r$ with $r:=\dist(S_i,\Omega/\Omega_i')/3$ and $|\nabla \zeta|\leq C /r$,
we obtain that
\begin{equation}\label{ksjdkddshdjhed}
\int_{\Omega/\Omega_i'} |\nabla \psi_{i,T}|^2 \leq \frac{C}{r^2} \int_{(\Omega/\Omega_i')^r} \psi_{i,T}^2.
\end{equation}
Next, observe that for $x\in (\Omega/\Omega_i')^r$,
\begin{equation}\label{ksjdkseddrsdrhed}
\psi_{i,T}(x)=-\int_{S_i} \nabla G_{T,\Omega}(x,y) \nabla \varphi_i (y)\,dy.
\end{equation}
Hence,
\begin{equation}\label{ksjdksddsxxrsdrhed}
\big|\psi_{i,T}(x)\big| \leq \|\nabla \varphi_i\|_{(L^2(S_i))^d} \big\|\nabla G_{T,\Omega}(x,.) \big\|_{(L^2(S_i))^d}.
\end{equation}
Another use of Cacciopoli's inequality leads to
\begin{equation}\label{kjshdjdwedrhkdhdje}
\big\|\nabla G_{T,\Omega}(x,.) \big\|_{(L^2(S_i))^d} \leq \frac{C}{r} \big\| G_{T,\Omega}(x,.) \big\|_{L^2(S_i^r)}.
\end{equation}
Combining \eqref{ksjdkddshdjhed} with \eqref{ksjdksddsxxrsdrhed} with \eqref{kjshdjdwedrhkdhdje}, we obtain that
\begin{equation}\label{ksjdkddsesddshdjhed}
\int_{\Omega/\Omega_i'} |\nabla \psi_{i,T}|^2 \leq  \|\nabla \varphi_i\|_{(L^2(S_i))^d}^2 \frac{C}{r^4} \int_{(\Omega/\Omega_i')^r} \big\| G_{T,\Omega}(x,.) \big\|_{L^2(S_i^r)}^2.
\end{equation}
We conclude the proof of \eqref{contpsiext1} using Lemma \ref{controlG} and \eqref{Controlphi}. The proof of \eqref{contpsiext2} is similar observing that $\dist(S_i,\Omega/\Omega_i')\leq \dist(S_i,\Omega_i/\Omega_i')$
\end{proof}

\begin{Lemma}\label{diffLem}
Let $D$ be a sub-domain of $\Omega$ with piecewise Lipschitz boundary. Let $\psi\in H^1(\Omega)$, and let $v$ solve
\begin{equation}\label{scadeqrddbdleqm0Cac}
\begin{cases}
    \frac{v}{T}-\diiv \Big(a(x)  \nabla v(x)\Big)=0 \quad  x \in D,\\
    v=\psi \quad \text{on}\quad \partial D,
    \end{cases}
\end{equation}
Write $S$ the intersection of the support of $\psi$ with $D$. Let $D_1$ be a sub-domain of $D$ such that $\dist(D_1,S)>0$, then
\begin{equation}
\int_{D_1} |\nabla v|^2 \leq  \frac{C}{\big(\dist(D_1,S)\big)^{2d}} (T^{-1}+1)^2 \|\psi\|_{H^1(\Omega)}^2 \exp\Big(-\frac{\dist(D_1,S)}{C \sqrt{T}}\Big),
\end{equation}
where $C$ does not depend on $D,D_1,S$.
\end{Lemma}
\begin{proof}
Write $w:=v-\psi$. Then,
\begin{equation}\label{scaqdzerddabdlem0Cac}
\begin{cases}
    \frac{w}{T}-\diiv \Big(a(x)  \nabla w(x)\Big)=-\frac{\psi}{T}+\diiv(a\nabla \psi) \quad  x \in D,\\
    v=0 \quad \text{on}\quad \partial D,
    \end{cases}
\end{equation}
Thus,
\begin{equation}\label{scadswezerwddbdlem0pCac}
w(x)=-\int_D \big(\frac{\psi(y)}{T}G_{T,D}(x,y)+ \nabla \psi(y) a(y) \nabla G_{T,D}(x,y)\big)\,dy.
\end{equation}
Using Cauchy-Schwartz inequality, we obtain that
\begin{equation}\label{scadswezewrwddbdlecm0Cac}
|w(x)|\leq C \|\psi\|_{H^1(\Omega)} \Big(\frac{1 }{T}\big\|G_{T,D}(x,.)\big\|_{L^2(S)}+ \big\|\nabla G_{T,D}(x,.)\big\|_{(L^2(S))^d}\Big).
\end{equation}
For $A\subset D$, write $A^r$ the set of points of $D$ that are at distance at most $r$ from $A$.
Let us now use Cacciopoli's inequality to bound $\int_{D_1} |\nabla w|^2$.  Using Lemma \ref{LemCaccio} with  $\zeta$ identically equal to one on $D_1$, zero on $D/ D_1^{r_1}$ and such that $|\nabla \zeta|\leq C/r_1$  we obtain that
\begin{equation}\label{ksjdkshdjhed}
\int_{D_1} |\nabla w|^2 \leq \frac{C}{r_1^2} \int_{D_1^{r_1}} w^2,
\end{equation}
provided that $\dist(D_1^{r_1} ,S)>0$. Hence, for $r_1:=\dist(D_1,S)/3$, we obtain \eqref{ksjdkshdjhed}.
Taking $r_2:=\dist(D_1,S)/3$ and using Cacciopoli's inequality again, we also obtain that
 \begin{equation}\label{kjsdhgjksghh}
\big\|\nabla G_{T,D}(x,.)\big\|_{(L^2(S))^d} \leq \frac{C}{r_2} \big\|G_{T,D}(x,.)\big\|_{L^2(S^{r_2})}.
\end{equation}
Combining \eqref{ksjdkshdjhed} with \eqref{scadswezewrwddbdlecm0Cac} and \eqref{kjsdhgjksghh} and observing that $w=v$ on $D_1^{r_1}$ we obtain that
\begin{equation}\label{ksjsdkeewdkshdjdhed}
\int_{D_1} |\nabla v|^2 \leq \frac{C}{r_1^2 r_2^2} \|\psi\|_{H^1(\Omega)}^2 (T^{-1}+1)^2 \int_{D_1^{r_1}} \big\|G_{T,D}(x,.)\big\|_{L^2(S^{r_2})}^2.
\end{equation}
Using Lemma \ref{controlG}, we deduce that
\begin{equation}\label{ksjsdkedhksderhhfdkshdjhed}
\int_{D_1} |\nabla w|^2 \leq \frac{C |\Omega|}{(\dist(D_1,S))^{2d}} \|\psi\|_{H^1(\Omega)}^2 (T^{-1}+1)^2 \exp\big(-\frac{\dist(D_1,S)}{C\sqrt{T}}\big).
\end{equation}
This concludes the proof of Lemma \ref{diffLem}.
\end{proof}

\begin{Lemma}\label{hsgsjhgsjg3e3ee}
Let $\psi_{i,T}$ be the solution of \eqref{loclakjlkajlei23TT} and $\psi_{i,T,\Omega_i}$ the solution of \eqref{loclakjlddrfredkajlei23TT}.
Let $\Omega_i'$ be a sub-domain of $\Omega_i$ such that  $\dist(\Omega/\Omega_i, \Omega_i')>0$. We have
\begin{equation}\label{contpssdsddeeiext1}
\big\|\psi_{i,T}-\psi_{i,T,\Omega_i}\big\|_{H^1(\Omega_i')} \leq \frac{C (T^{-1}+1) h^{\frac{d}{2}-1}}{\big(\dist(\Omega/\Omega_i,\Omega_i')\big)^{d+1}} \exp\Big(-\frac{\dist(\Omega/\Omega_i, \Omega_i')}{C \sqrt{T}}\Big).
\end{equation}
\end{Lemma}
\begin{proof}
Lemma \ref{hsgsjhgsjg3e3ee} is a direct consequence of Lemma \ref{diffLem}. To this end, we choose $D:=\Omega_i$, $v=:\psi_{i,T}-\psi_{i,T,\Omega_i}$ and $D_1:=\Omega_i'$. We also choose $\psi:=\eta\psi_{i,T}$ where $\eta: \Omega \rightarrow [0,1]$ is $C^1$, equal to one on
$\Omega/\Omega_i$ and $0$ on $(\Omega/\Omega_i)^r$ with $r:=\dist(\Omega/\Omega_i, \Omega_i')/3$ ($A^r$ being the set of points in $\Omega$ at distance at most $r$ from $A$) and $|\nabla \eta|\leq C/r$. We obtain from Lemma \ref{diffLem} that
\begin{equation}\label{contpsssedsdeiext1}
\big\|\psi_{i,T}-\psi_{i,T,\Omega_i}\big\|_{H^1(\Omega_i')} \leq \frac{C (T^{-1}+1) }{\big(\dist(\Omega/\Omega_i,\Omega_i')\big)^{d}} \|\psi\|_{H^1(\Omega)} \exp\Big(-\frac{\dist(\Omega/\Omega_i, \Omega_i')}{C \sqrt{T}}\Big).
\end{equation}
We conclude using \eqref{Controlphi} and $\|\psi\|_{H^1(\Omega)} \leq \frac{C}{\dist(\Omega/\Omega_i, \Omega_i')} \|\nabla \varphi_i\|_{(L^2(\Omega))^d}$.
\end{proof}
Observing that
\begin{equation}
\big\|\psi_{i,T}-\psi_{i,T,\Omega_i}\big\|_{H^1(\Omega)} \leq \big\|\psi_{i,T}-\psi_{i,T,\Omega_i}\big\|_{H^1(\Omega_i')}+
\big\|\psi_{i,T}\big\|_{H^1(\Omega/\Omega_i')}+\big\|\psi_{i,T,\Omega_i}\big\|_{H^1(\Omega_i/\Omega_i')},
\end{equation}
we conclude the proof of Proposition \ref{jhgsjhgdg6354r5df} by using Lemma \ref{hsgsjhgsjg3e3ee} and Lemma \ref{sumLemma}
with $\Omega_i':= S_i^r$ where $S_i^r$ are the points in $\Omega_i$ at distance at most $r$ from $S_i$ with $r:=\dist(S_i,\Omega/\Omega_i)/3$.

\section{On the compactness of the solution space.}\label{compactness}
Although the foundations of classical homogenization \cite{BeLiPa78} were laid down based on assumptions of periodicity (or ergodicity) and scale separation, numerical homogenization, as described here, is independent from these concepts and solely relies on the strong compactness of the solution space (and the fact that a compact set can be covered with a finite number of balls of arbitrary sizes).
Observe that an analogous notion of compactness supports the foundations of  $G$ and $H$-convergence (\cite{Mur78}, \cite{MR0240443,Gio75}). The main difference is that $G$ and $H$-convergence  rely on pre-compactness and weak convergence of fluxes and here, we rely on compactness in the (strong) $H^1_0$-norm, i.e. the following theorem.

Let $W$ be the range of $g$ in \eqref{scalarproblem0}. Write

\begin{equation}
V:=\{u\in H^1_0(\Omega)\,:\, u\text{ solves \eqref{scalarproblem0} for some } g\in W\}.
\end{equation}

\begin{Theorem}
Let $\nu<1$. If $W$ is a closed bounded subset of $H^{-\nu}(\Omega)$ then $W$ is a compact subset of $H^1_0(\Omega)$ (in the strong $H^1_0$-norm).
\end{Theorem}
\begin{proof}
 We have
$(a\nabla u)_{pot}=-\nabla \Delta^{-1} g$. So using the same notation as
 in \eqref{ksjjseddesel3} we get $ (a \nabla V)_{pot}= -\nabla \Delta^{-1}W$.
 Let $u_n$ be a sequence in $V$ then there exists a sequence in $W$ such that $-{\rm div}(a\nabla u_n)=g_n$. Using the fact that
$-\nabla \Delta^{-1}W$ is a compact subset of $(L^2(\Omega))^d$ (we refer, for instance, to the Kondrachov embedding theorem) we get that there exists $g^*\in W$ such that $\|\nabla \Delta^{-1}g_n - \nabla \Delta^{-1}g^*\|_{L^2}\rightarrow 0$. Writing $u^*$ the solution of  $-{\rm div}(a\nabla u^*)=g^*$ and using
$(a\nabla (u_n-u^*))_{pot}=-\nabla \Delta^{-1} (g_n-g^*)$ we get that
$\|(a\nabla (u_n-u^*))_{pot}\|_{L^2}\rightarrow 0$. Using the equivalence between the flux norm and the $H^1_0$ norm we deduce that
$\|u_n-u^*\|_{H^1_0}\rightarrow 0$. This finishes the proof.
\end{proof}

This notion of compactness of the solution space constitutes a simple but fundamental link between classical homogenization, numerical homogenization
and reduced order modeling (or reduced basis modeling \cite{MR2359667, MR1781533}) (we also refer to the discussion in Section 6 of \cite{BerlyandOwhadi10}).
This notion is also what allows for atomistic to continuum up-scaling \cite{ZhBFOww09},
the basic idea is that if source (force) terms are integrable enough (for instance in $L^2$ instead of $H^{-1}$) then the solution space is no longer $H^1$ but a sub-space $V$ that is compactly embedded into $H^1$ and, hence, it can be approximated by a finite-dimensional space (in $H^1$-norm).
In other words if these systems are ``excited'' by ``regular'' forces or source terms (think compact, low dimensional) then the solution space can be approximated by a low dimensional space (of the whole space) and the name of the game becomes ``how to approximate'' this solution space (and this can be done by using local time-independent solutions).

\paragraph{Acknowledgements.}
We thank L. Berlyand for stimulating discussions. We also thank Ivo Babu\v{s}ka, John Osborn, George Papanicolaou and Bj\"{o}rn Engquist for pointing us in the direction of the localization problem. The work of H. Owhadi is partially supported by the National Science Foundation under Award Number  CMMI-092600 and the Department of Energy National Nuclear Security Administration under Award Number DE-FC52-08NA28613. The work of L. Zhang is partially supported by the EPSRC Science and Innovation award to the Oxford Centre for Nonlinear PDE (EP/E035027/1).
We thank Sydney Garstang for proofreading the manuscript.

\def\cprime{$'$} \def\cprime{$'$}
  \def\polhk#1{\setbox0=\hbox{#1}{\ooalign{\hidewidth
  \lower1.5ex\hbox{`}\hidewidth\crcr\unhbox0}}} \def\cprime{$'$}
  \def\polhk#1{\setbox0=\hbox{#1}{\ooalign{\hidewidth
  \lower1.5ex\hbox{`}\hidewidth\crcr\unhbox0}}}

\end{document}